\documentclass[11pt,leqno]{amsart}

\usepackage{amsmath}
\usepackage{amssymb}
\usepackage{a4wide}
\usepackage{bbm}
\usepackage{graphics}
\usepackage{epsfig}
\usepackage{color}
\usepackage{mathrsfs}
\usepackage[sans]{dsfont}

\usepackage{hyperref}

\newcommand{\Subsection}[1]{\subsection{ #1} ${}^{}$}

\newtheorem{theorem}{Theorem}[section]
\newtheorem{lemma}[theorem]{Lemma}
\newtheorem{proposition}[theorem]{Proposition}

\newcounter{hypo}
\newcounter{hypoa}

\newenvironment{hyp}{  \begin{enumerate} \setcounter{enumi}{\value{hypo}} \item}{\stepcounter{hypo} \end{enumerate}}

\def\C{{\mathbb C}}

\def\N{{\mathbb N}} 
\def\R{{\mathbb R}} 

\def\Z{{\mathbb Z}}

\def\CE{\mathcal {E}}

\def\CH{\mathcal {H}}

\def\CO{\mathcal {O}}

\def\CR{\mathcal {R}}
\def\CS{\mathcal {S}}

\def\SA{\mathscr {A}}
\def\SB{\mathscr {B}}

\def\SF{\mathscr {F}}
\def\SG{\mathscr {G}}

\def\ker{\mathop{\rm Ker}\nolimits}

\def\one{\mathds{1}}
\def\re{\mathop{\rm Re}\nolimits}
 \def\im{\mathop{\rm Im}\nolimits}

\def\Op{\mathop{\rm Op}\nolimits}

\newcommand{\divergence}{\operatorname{div}}

\newcommand{\supp}{\operatorname{supp}}

\def\arg{\mathop{\rm arg}\nolimits}
\def\dist{\mathop{\rm dist}\nolimits}
\def\diag{\mathop{\rm diag}\nolimits}

\def\<{\langle}
\def\>{\rangle}

\def\res{\mathop{\rm Res}\nolimits}

\def\MS{\mathop{\rm MS}\nolimits}

\makeatletter
 \@addtoreset{equation}{section}
 \makeatother
 
\title{Barrier-top resonances for non globally analytic potentials}

\author[J.-F. Bony]{Jean-Fran\c{c}ois Bony}
\address{Jean-Fran\c{c}ois Bony, IMB, CNRS (UMR 5251), Universit\'e de Bordeaux, 33405 Talence, France}
\email{bony@math.u-bordeaux.fr}
\author[S. Fujii\'e]{Setsuro Fujii\'e}
\address{Setsuro Fujii\'e, Department of Mathematical Sciences, Ritsumeikan University, 1-1-1 Noji-Higashi, Kusatsu, 525-8577 Japan}
\email{fujiie@fc.ritsumei.ac.jp}
\author[T. Ramond]{Thierry Ramond}
\address{Thierry Ramond, Laboratoire de Math\'ematiques d'Orsay, Univ. Paris-Sud, CNRS, Universit\'e Paris-Saclay, 91405 Orsay, France}
\email{thierry.ramond@math.u-psud.fr}
\author[M. Zerzeri]{Maher Zerzeri}
\address{Maher Zerzeri, Universit\'e Paris 13, Sorbonne Paris Cit\'e, LAGA, CNRS (UMR 7539), 93430 Villetaneuse, France}
\email{zerzeri@math.univ-paris13.fr}


\thanks{\textbf{Acknowledgments:} The second author was partially supported by the JSPS KAKENHI Grant 15K04971.}

\begin{document}

\begin{abstract}
We give the semiclassical asymptotic of barrier-top resonances for Schr\"{o}dinger operators on $\R^{n}$, $n \geq 1$, whose potential is $C^{\infty}$ everywhere  and analytic at infinity. In the globally analytic setting, this has already been obtained in \cite{BrCoDu87_02,Sj87_01}. Our proof is based on a propagation of singularities theorem at a hyperbolic fixed point that we establish here. This last result refines a theorem of \cite{BoFuRaZe07_01}, and its proof follows another approach.
\end{abstract}

\maketitle


\section{Introduction} \label{s1}

In this paper, we consider mainly Schr\"{o}dinger operators $P$ on $L^{2} ( \R^{n} )$, $n \geq 1$,
\begin{equation}\label{a56}
P = - h^{2} \Delta + V ( x ),
\end{equation}
where $V$ is a real-valued smooth function, which is supposed to be analytic outside of a compact set and to vanish at infinity. The resonances of $P$ near the real axis are thus well-defined through the analytic distortion method due to Aguilar and Combes \cite{AgCo71_01} and Hunziker \cite{Hu86_01}, and we denote $\res ( P )$ their set.

We are interested in resonances of $P$ generated by a unique non-degenerate global maximum of $V$. We can suppose that it is at $0$ and that 
\begin{equation}
V ( x ) = E_{0} - \sum_{j = 1}^{n} \frac{\lambda_{j}^{2}}{4} x_{j}^{2} + \CO ( x^{3} ) , 
\end{equation}
with $E_{0} > 0$ and $0 < \lambda_{1} \leq \dots \leq \lambda_{n}$. In fact, our precise assumption is that the trapped set at energy $E_{0}$, that is the set of bounded classical trajectories with energy $E_{0}$, is $\{ ( 0 , 0 ) \}$.

When $V$ is globally analytic, this question has been solved by Sj\"{o}strand \cite{Sj87_01} in the general case of a non-degenerate critical point, and by Briet, Combes and Duclos \cite{BrCoDu87_02} under a virial assumption (see also the third author \cite{Ra96_01} in dimension $1$). They have proved that, in any complex neighborhood of $E_0$ of size $h$, there is a bijection $b_{h}$ between $\res_{0} ( P )$ and $\res ( P )$, with $b_{h} ( z ) = z + o ( h )$ and 
\begin{equation}\label{a57}
\res_{0} ( P ) = \bigg\{ E_{0} - i h \sum_{j = 1}^{n} \lambda_{j} \Big( \frac{1}{2} + \alpha_{j} \Big) ; \ \alpha \in \N^{n} \bigg\} ,
\end{equation}
is what we call the set of pseudo-resonances. They obtained further a full asymptotic expansion for isolated resonances. Then, under the assumption that the $\lambda_j$'s are $\Z$-independent, Kaidi and Kerdelhu{\'e} \cite{KaKe00_01} have extended this result to disks of size $h^\delta$ for any $\delta \in ]0,1]$. Eventually, Hitrik, Sj\"ostrand and V{\~u} Ng{\d{o}}c \cite{HiSjVu07_01} have obtained a similar result in dimension 2 for disks of size $1$. This kind of resonances appears naturally in the context of general relativity, in particular for the study of black holes as in S{\'a} Barreto and Zworski \cite{SaZw97_01}.

In all these works, the potential was supposed to be analytic in a whole neighborhood of $\R^{n}$. We obtain here the same kind of result, with the weaker assumption that $V$ is $C^{\infty}$ everywhere and analytic at infinity. Our approach also provides a polynomial estimate for the cut-off resolvent away from the resonances, and we are able to describe the resonant states. However, we do not control the multiplicity of the resonances. The precise statements are given in Section \ref{s5}. Note that, for shape resonances, Lahmar-Benbernou, Martinez and the second author \cite{FuLaMa11_01} have already proved that one can relax the analyticity hypothesis in the work of Helffer and Sj{\"o}strand \cite{HeSj86_01}. Concerning the study of the scattering amplitude at a barrier-top, the analyticity of the potential is not necessary (see e.g. \cite{AlBoRa08_01,CoPaSc15_01}).

As a matter of fact, our result about resonances is closely related to the description of the propagation of singularities in a neighborhood of a hyperbolic fixed point, for energies $z$ close to the critical energy $E_{0}$. Indeed we have proved in Section 8 of \cite{BoFuRaZe16_01} that the absence of resonances follows from the uniqueness of the solution to the microlocal Cauchy problem at the trapped set
\begin{equation} \label{a70}
\left\{
\begin{aligned}
&( P - z ) u = v &&\text{microlocally near the trapped set,} \\
&u = u_{0} &&\text{microlocally in the region incoming from infinity.}
\end{aligned} 
\right.
\end{equation}
In this paper, we show uniqueness for \eqref{a70} when the energy $z$ is at distance $h$ of $\res_{0} ( P )$. The corresponding theorem can be found in Section \ref{s6}. In a previous work \cite{BoFuRaZe07_01}, such a result was obtained for energies $z$ outside of some exceptional set which was not completely described. In other words, we prove here that the propagation of singularities results of \cite{BoFuRaZe07_01} hold whenever the energy $z$ is not close to a pseudo-resonance.

The method we use here is quite different and less technical compared with that in \cite{BoFuRaZe07_01}. It is based on commutations with annihilation operators (as for the computation of the spectrum of the harmonic oscillator) rather than on energy estimates. For the sake of simplicity, let us explain this approach for $P = - h^{2} \Delta - x^{2}$ in dimension $1$. If we denote the annihilation-like operator $A = -i h \partial_{x} -x$, a direct computation gives
\begin{equation} \label{a71}
A ( P - z ) = ( P - 2 i h - z ) A .
\end{equation}
Consider now a solution $u$ of \eqref{a70} with $v = u_{0} = 0$ and $z \in B ( 0 , C h )$ with $B ( a , r ) = \{ z ; \ \vert z - a \vert < r \}$. Applying $m$ times \eqref{a71}, we get
\begin{equation}
( P - 2 m i h - z ) A^{m} u = A^{m} ( P - z ) u = \CO ( h^{\infty} ) ,
\end{equation}
for all $m \in \N$. Since the operator $P - 2 m i h - z$ is invertible for $m > C / 2$, we deduce $A^{m} u = \CO ( h^{\infty} )$ for these values of $m$. A standard argument of microlocal analysis implies that $u$ is a Lagrangian distribution associated to $\Lambda_{+}$, the characteristic manifold of the pseudodifferential operator $A$. A similar idea was used by Hassell, Melrose and Vasy \cite{HaMeVa04_01,HaMeVa08_01} in a different context. Using that $u$ is a Lagrangian distribution, \eqref{a70} becomes a transport equation on the symbol of $u$. Then, if $z$ is away from $\res_{0} ( P )$, we can conclude that $u = 0$ microlocally near $( 0 , 0 )$. This approach may be carried out in other contexts, even if only in the case of a general non-degenerate critical point. Each time, the notion of Lagrangian distribution associated to $\Lambda_{+}$ used above should be replaced by an appropriate class of functions.

For the reader's convenience, we recall some notations and terminology of microlocal analysis in Appendix \ref{s7}.

\section{Setting and results} \label{s2}

\Subsection{Asymptotic of barrier-top resonances} \label{s5}

We state here the asymptotic of the barrier-top resonances for non-globally analytic potentials. We consider the semiclassical Schr\"{o}dinger operator $P$ on $L^{2} ( \R^{n} )$ given in \eqref{a56} and we assume that
\begin{hyp} \label{h1}
$V \in C^{\infty} ( \R^{n} ; \R )$ extends holomorphically in the sector 
\begin{equation*}
\CS = \big\{ x \in \C^{n} ; \ \vert \re x \vert > C \text{ and } \vert \im x \vert \leq \delta \vert x \vert \big\} ,
\end{equation*}
for some $C , \delta > 0$. Moreover, $V ( x ) \to 0$ as $x \to \infty$ in $\CS$.
\end{hyp}
Under this assumption, one can define the distorted operator $P_{\theta}$ of angle $\theta > 0$ small enough. Its spectrum is discrete in $\CE_{\theta} = \{ z \in \C ; \ - 2 \theta < \arg z \leq 0 \}$, and the resonances of $P$ are the eigenvalues of $P_{\theta}$ in $\CE_{\theta}$. They can also be defined as the poles of the meromorphic extension of $( P - z )^{- 1} : L^{2}_{\rm comp} ( \R^{n} ) \to L^{2}_{\rm loc} ( \R^{n} )$ from the upper complex half-plane. We send back the reader to Section 2 of \cite{BoFuRaZe16_01} (and the references given there) for the precise definitions and more details on the resonances.

We denote $p ( x , \xi ) = \xi^{2} + V ( x )$ the symbol of $P$. The associated Hamiltonian vector field is
\begin{equation*}
H_{p} = \partial_{\xi} p \cdot \partial_{x} - \partial_{x} p \cdot \partial_{\xi} = 2 \xi \cdot \partial_{x} - \nabla V ( x ) \cdot \partial_{\xi} .
\end{equation*}
Integral curves $t \mapsto \exp ( t H_{p} )( x , \xi )$ of $H_{p}$ are called Hamiltonian or bicharacteristic trajectories, and $p$ is constant along such curves. We also suppose that
\begin{hyp} \label{h2}
$V$ has a non-degenerate maximum at $x=0$ and
\begin{equation*}
V (x) = E_{0} - \sum_{j = 1}^{n} \frac{\lambda_{j}^{2}}{4} x_{j}^{2} + \CO ( x^{3} ) ,
\end{equation*}
with $E_{0}> 0$ and $0 < \lambda_{1} \leq \cdots \leq \lambda_{n}$.
\end{hyp}
We define the trapped set at energy $E$ for $P$ as
\begin{equation*}
K ( E ) = \big\{ ( x , \xi ) \in p^{- 1} ( E ) ; \ t \mapsto \exp ( t H_{p} ) ( x , \xi ) \text{ is bounded} \big\} .
\end{equation*}
For $E > 0$, $K ( E )$ is compact and stable by the Hamiltonian flow. We finally assume that
\begin{hyp} \label{h3}
The trapped set at energy $E_{0}$ is $K ( E_{0} ) = \{ ( 0,0) \}$.
\end{hyp}
In particular, $x = 0$ is the unique global maximum for $V$. Moreover, there exists a pointed neighborhood of $E_{0}$ in which all the energy levels are non trapping.

The previous assumptions imply that $( 0 , 0 )$ is a hyperbolic fixed point for $H_{p}$. The stable/unstable manifold theorem ensures the existence of the incoming Lagrangian manifold $\Lambda_{-}$ and the outgoing Lagrangian manifold $\Lambda_{+}$ characterized by
\begin{equation*}
\Lambda_{\pm} = \big\{ ( x , \xi ) \in T^{*} \R^{n} ; \ \exp ( t H_{p} ) ( x , \xi ) \to ( 0 , 0 ) \text{ as } t \to \mp \infty \big\} .
\end{equation*}
They are stable by the Hamiltonian flow and included in $p^{- 1} ( E_{0} )$. Eventually, there exist two smooth functions $\varphi_{\pm}$, defined in a vicinity of $0$, satisfying
\begin{equation} \label{a27}
\varphi_{\pm} ( x ) = \pm \sum_{j = 1}^{n} \frac{\lambda_{j}}{4} x_{j}^{2} + \CO ( x^{3} ) ,
\end{equation}
and such that $\Lambda_{\pm} = \Lambda_{\varphi_{\pm}} = \{ ( x , \xi ) ; \ \xi = \nabla \varphi_{\pm} (x) \}$ near 
$( 0 , 0 )$.

\begin{figure}
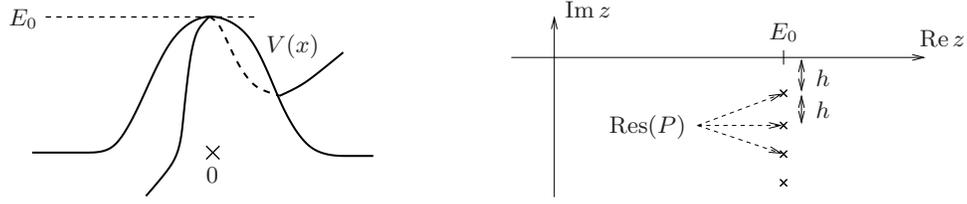

\begin{center}
\input{picture04.pstex_t} $\qquad \qquad$ \input{picture05.pstex_t}
\end{center}
\caption{A potential as in Theorem \ref{a37} and the corresponding resonances.} \label{f2}
\end{figure}

We recall that $\res ( P )$ is the set of resonances of $P$, and that $\res_{0} ( P )$ is the set of pseudo-resonances given in \eqref{a57}. For $A , B , C$ subsets of $\C$ and $\delta \geq 0$, we say that $\dist ( A , B ) \leq \delta$ in $C$ if and only if
\begin{align*}
&\forall a \in A \cap C , \quad \exists b \in B , \qquad \vert a - b \vert \leq \delta ,   \\
\text{and} \quad &\forall b \in B \cap C , \quad \exists a \in A , \qquad \vert a - b \vert \leq \delta .
\end{align*}
Concerning the barrier-top resonances, our main result is the following.

\begin{theorem}[Asymptotic of resonances]\sl \label{a37}
Assume \ref{h1}--\ref{h3} and let $C > 0$. In the domain $B ( E_{0} , C h )$, we have
\begin{equation*}
\dist \big( \res ( P ) , \res_{0} ( P ) \big) = o ( h ) ,
\end{equation*}
as $h$ goes to $0$. Moreover, for all $\chi \in C^{\infty}_{0} ( \R^{n} )$ and $\varepsilon > 0$, there exists $M > 0$ such that
\begin{equation} \label{a38}
\big\Vert \chi ( P -z )^{-1} \chi \big\Vert \leq h^{- M} ,
\end{equation}
uniformly for $h$ small enough and $z \in B ( E_{0} , C h ) \setminus ( \res_{0} ( P ) + B ( 0 , \varepsilon h ) )$.
\end{theorem}

The distribution of the (pseudo-)resonances is illustrated in Figure \ref{f2}. From Proposition D.1 of \cite{BoFuRaZe16_01}, the inequality \eqref{a38} is equivalent to a polynomial estimate of the distorted resolvent, that is $\Vert ( P_{\theta} - z )^{- 1} \Vert \leq h^{- N}$ with $\theta = h \vert \ln h \vert$.

The previous theorem provides the asymptotic distribution of resonances as a set, but gives no bijection between resonances and pseudo-resonances, as it was the case in \cite{BrCoDu87_02,Ra96_01,Sj87_01}. In other words, there may be more than one resonance near each pseudo-resonance (and reciprocally). That the multiplicity is unknown comes from our method of proof. Nevertheless, following the proof of Proposition 4.2 of \cite{BoFuRaZe07_01} (see also Section 3 of \cite{BoFuRaZe11_01}), one can show that the number of resonances counted with their multiplicity near each pseudo-resonance is uniformly bounded. It should also be possible to prove that number of resonances is greater or equal to the multiplicity of pseudo-resonances as in Proposition 4.6 of \cite{BoFuRaZe16_01}.

The results of this part are stated for Schr\"{o}dinger operators \eqref{a56} but hold true in more general settings. What is really needed is that the resonances of the pseudodifferential operator $P = \Op ( p )$ can be defined by complex distortion and that the trapped set at energy $E_{0}$ consists of a hyperbolic fixed point. Indeed, the results used in the proofs (that is Section 5 of \cite{BoFuRaZe07_01} and Theorem \ref{a1} below) are valid for pseudodifferential operators. In particular, if $P$ has a subprincipal symbol, say $p ( x, \xi , h ) = p_{0} ( x , \xi ) + h p_{1} ( x , \xi , h )$, then the set of pseudo-resonances \eqref{a57} should be replaced by
\begin{equation*}
\res_{0} ( P ) = \bigg\{ E_{0} - i h \sum_{j = 1}^{n} \lambda_{j} \Big( \frac{1}{2} + \alpha_{j} \Big) + h p_{1} ( 0 , 0 , h ) ; \ \alpha \in \N^{n} \bigg\} .
\end{equation*}
Note that the ``black box'' framework of Sj\"{o}strand and Zworski \cite{SjZw91_01} allows to define the resonances for a large class of operators.

Theorem \ref{a37} provides the asymptotic of the resonances modulo $o ( h )$. Under certain circumstances, it is possible to prove that they have an asymptotic expansion in powers of $h$. In this direction, we state the following result which is similar to Proposition 0.3 of Sj\"{o}strand \cite{Sj87_01}.

\begin{proposition}[Asymptotic modulo $\CO ( h^{\infty} )$]\sl \label{a58}
In the setting of Theorem \ref{a37}, let $\alpha \in \N^{n}$ be such that the corresponding element $z_{0} ( h ) = E_{0} - i h \sum_{j} \lambda_{j} ( \alpha_{j} + 1 / 2 )$ of $\res_{0} ( P )$ is simple. Then, there exist $\delta > 0$ and $z_{\infty} ( h )$ satisfying $z_{\infty} ( h ) \simeq E_{0} + E_{1} h + E_{2} h^{2} + \cdots$ as an asymptotic expansion with $E_{1} = - i \sum_{j} \lambda_{j} ( \alpha_{j} + 1 / 2 )$ and
\begin{equation*}
\dist \big( \res ( P ) , \{ z_{\infty} ( h ) \} \big) = \CO ( h^{\infty} ) ,
\end{equation*}
in $B ( z_{\infty} ( h ) , \delta h )$. Moreover, for all $N > 0$ and $\chi \in C^{\infty}_{0} ( \R^{n} )$, there exists $M > 0$ such that
\begin{equation*}
\big\Vert \chi ( P - z )^{-1} \chi \big\Vert \leq h^{- M} ,
\end{equation*}
uniformly for $h$ small enough and $z \in B ( z_{\infty} ( h ) , \delta h ) \setminus B ( z_{\infty} ( h ) , h^{N} )$.
\end{proposition}

Here, we say that $z_{0} ( h )$ is simple when there is a unique $\alpha \in \N^{n}$ such that
$z_{0} ( h ) = E_{0} - i h \sum_{j} \lambda_{j} ( \alpha_{j} + 1 / 2 )$.
Note that the asymptotic expansion of $z_{\infty} ( h )$ is only known through an implicit relation (see \eqref{a62}).

As a byproduct of the proof of Theorem \ref{a37}, we obtain the asymptotic of the resonant states. Recall that, for $z \in \CE_{\theta} \cap \res ( P )$ and $u \in H^{2} ( \R^{n} )$, we say that $u$ is a resonant state of $P$ associated to the resonance $z$ if and only if $( P_{\theta} - z ) u = 0$. We send back the reader to Section 7 of \cite{BoFuRaZe16_01} and the references given there for more details.

\begin{proposition}[Description of the resonant states]\sl \label{a66}
In the setting of Theorem \ref{a37}, we fix $\theta = h \vert \ln h \vert$. Let $u = u ( h )$ be a family of normalized resonant states associated to a resonance $z = z ( h ) \in B ( E_{0} , C h )$. Then, there exists $a \in S ( h^{- M} )$ with $M \in \R$ such that
\begin{equation*}
u ( x, h ) = a ( x , h ) e^{i \varphi_{+} ( x ) / h} \text{ microlocally near } ( 0 , 0 ) .
\end{equation*}
Moreover, the symbol $a$ satisfies near $0$ the transport equation
\begin{equation} \label{a67}
2 \nabla \varphi_{+} ( x ) \cdot \nabla a ( x , h ) + \Big( \Delta \varphi_{+} ( x ) - i \frac{z - E_{0}}{h} \Big) a ( x , h ) - i h \Delta a ( x , h ) \in S ( h^{\infty} ) .
\end{equation}
\end{proposition}

One can also show that $a$ is not $\CO ( h^{\infty} )$ near $0$. For simple pseudo-resonances, the transport equation \eqref{a67} can be used to prove that $a$ has an asymptotic expansion in powers of $h$ whose leading term has an explicit behavior at $0$. Proposition \ref{a66} is a consequence of Lemma \ref{a28} and \eqref{a31} in the Schr\"{o}dinger case (see \eqref{a59}). For globally analytic potentials and simple pseudo-resonances, the generalized spectral projection (and then the resonant states) has already been described (see Theorem 4.1 of \cite{BoFuRaZe11_01}).

The proof of the above results can be found in Section \ref{s4}. It mainly rests on the propagation of singularities at the hyperbolic fixed point given in the next part and on the construction of test functions made in \cite{BoFuRaZe11_01}.

\Subsection{Propagation of singularities at barrier-top} \label{s6}

In this part, we give a result on the propagation of semiclassical singularities through a hyperbolic fixed point. We generalize slightly the setting of Section \ref{s5} and work microlocally near $( 0 , 0 )$. More precisely,
\begin{hyp}\label{h4}
Let $P = \Op ( p )$ on $L^{2} ( \R^{n} )$ with
\begin{equation*}
p ( x , \xi , h ) = p_{0} ( x , \xi ) + h p_{1} ( x , \xi , h ) ,
\end{equation*}
$p , p_{0} , p_{1} \in S ( 1 )$ and $p_{0}$ real-valued.
\end{hyp}
As in \ref{h2}, we assume that
\begin{hyp}\label{h5}
Up to a symplectic change of variables, we have
\begin{equation*}
p_{0} ( x , \xi ) = \xi^{2} - \sum_{j = 1}^{n} \frac{\lambda_{j}^{2}}{4} x_{j}^{2} + \CO \big( ( x , \xi )^{3} \big) ,
\end{equation*}
in a neighborhood of $( 0 , 0 )$, with $0 < \lambda_{1} \leq \cdots \leq \lambda_{n}$.
\end{hyp}
In that case, the exceptional set $\Gamma_{0} ( h )$ is defined by
\begin{equation}
\Gamma_{0} ( h ) = \bigg\{ - i h \sum_{j = 1}^{n} \lambda_{j} \Big( \frac{1}{2} + \alpha_{j} \Big) + h p_{1} ( 0 , 0 , h ) ; \ \alpha \in \N^{n} \bigg\} ,
\end{equation}

\begin{figure}
\begin{center}
\input{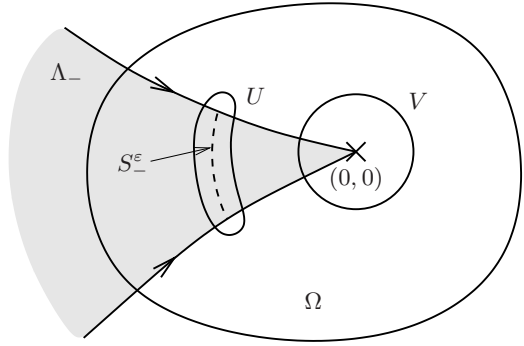}
\end{center}
\caption{The geometric setting of Theorem \ref{a1}.} \label{f3}
\end{figure}

\begin{theorem}[Propagation of singularities]\sl \label{a1}
Assume \ref{h4} and \ref{h5}. Let $\Omega \subset T^{*} \R^{n}$ be a neighborhood of $( 0 ,0 )$ and $S^{\varepsilon}_{-} = \{ ( x, \xi ) \in \Lambda_{-} ; \ \vert x \vert = \varepsilon \} \subset \Omega$, with $\varepsilon > 0$ sufficiently small. Let also $C , \delta > 0$ and $U$ be a neighborhood of $S^{\varepsilon}_{-}$. There exists a neighborhood $V$ of $( 0 , 0 )$ such that, for all $u \in L^{2} ( \R^{n} )$ satisfying $\Vert u \Vert_{L^{2}} \leq 1$ and
\begin{equation} \label{a69}
\left\{ \begin{aligned}
&( P - z ) u = 0 &&\text{ microlocally near } \Omega , \\
&u = 0 &&\text{ microlocally near } U ,
\end{aligned} \right.
\end{equation}
with $z \in B ( 0 , C h )$ and $\dist ( z , \Gamma_{0} ( h ) ) \geq \delta h$, then $u = 0$ microlocally near $V$.
\end{theorem}

In the previous result, $u = u ( h )$ and $z = z ( h )$ do not need to be defined for all positive $h$ small enough, but only on a sequence of positive $h$'s converging to $0$. The different sets appearing in Theorem \ref{a1} are illustrated in Figure \ref{f3}.

We have already obtained a similar result in a previous paper. More precisely, Theorem 2.1 of \cite{BoFuRaZe07_01} shows the uniqueness of the microlocal Cauchy problem \eqref{a69} when the spectral parameter $z \in B ( 0 , C h )$ satisfies $\dist ( z , \Gamma ( h ) ) > h^{N}$ for any $N > 0$. The discrete set $\Gamma ( h )$ was not explicitly given but verifies some properties. Thus, the present Theorem \ref{a1} says that, roughly speaking, $\Gamma ( h ) = \Gamma_{0} ( h )$ modulo $o ( h )$. In particular, $\Gamma ( h )$ can be replaced by $\Gamma_{0} ( h )$ in all the statements of \cite{BoFuRaZe16_01}. We assume here that $z$ is at distance $\delta h$ from $\Gamma_{0} ( h )$ and not $h^{N}$ since $\Gamma_{0} ( h )$ is an approximation modulo $o ( h )$ of the true exceptional set. One could probably replace $\delta h$ by $h^{1 + \nu}$ with $0 < \nu \ll 1$ (and $h^{N}$ for any $N > 0$ in a framework similar to Proposition \ref{a58}).

A self-contained demonstration of Theorem \ref{a1}, summarized at the end of the introduction, is given in Section \ref{s3}. Using annihilation operators, it does not follow the one of Theorem 2.1 of \cite{BoFuRaZe07_01}. We could also have proved this result from Theorem \ref{a37}.

\section{Proof of the propagation of singularities} \label{s3}

\begin{figure}
\begin{center}
\input{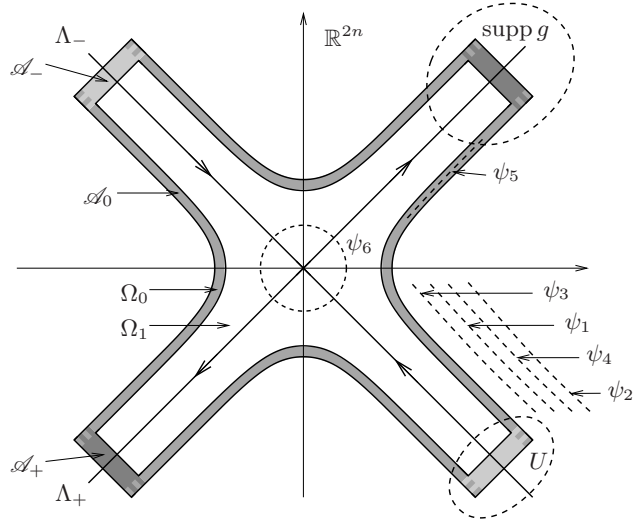}
\end{center}
\caption{The geometric setting and the different functions of Section \ref{s3}.} \label{f1}
\end{figure}

The beginning of the proof is similar to \cite[Section 3]{BoFuRaZe07_01}. Let $\Omega$ be a neighborhood of $(0,0)$. For $\varepsilon > 0$ small enough and $U$ a neighborhood of $S_{-}^{\varepsilon}$, one can find $\Omega_{1} \subset \Omega_{0} \subset \Omega$ close to $( 0 , 0 )$ as in Figure \ref{f1} with
\begin{equation*}
\Omega_{0} \setminus \Omega_{1} = \SA_{-} \cup \SA_{0} \cup \SA_{+} ,
\end{equation*}
satisfying the following properties. First $\Omega_{0} , \SA_{-} , \SA_{0} , \SA_{+}$ are open sets and $\Omega_{1}$ is a closed set. Then, $\SA_{-} \subset U$ and $\SA_{0}$ is geometrically controlled by $\SA_{-}$. It means that any point $\rho \in \SA_{0}$ can be written as $\exp ( t H_{p_{0}} ) ( \rho_{-} )$ for some $\rho_{-} \in \SA_{-}$ and some $t \geq 0$ with $\exp ( s H_{p_{0}} ) ( \rho_{-} ) \in \Omega_{0}$ for all $s \in [ 0 , t ]$. Eventually, $\SA_{+}$ is close to $\Lambda_{+}$. Note that the $\SA_{\bullet}$ are not disjoint. The existence of such a framework is guaranteed by the Hartman--Grobman theorem (see \cite[Page 120]{Pe01_01}).

Let $\psi_{1} , \psi_{2} \in C^{\infty}_{0} ( \R^{2 n} ; [ 0 , 1 ] )$ be supported near $\Omega_{0}$ and satisfy $\one_{\Omega_{0}} \prec \psi_{1} \prec \psi_{2}$. Here, $f \prec g$ means that $g = 1$ near the support of $f$. As explained in \cite[Lemma 2.1]{HeSj85_01} or \cite[(12.104)--(12.105)]{BoFuRaZe16_01}, there exists a local symplectic diffeomorphism $( x , \xi ) \mapsto ( y , \eta )$ such that, in these new variables, the principal symbol $p_{0}$ can be written
\begin{equation*}
p_{0} ( y , \eta ) = \SB ( y , \eta ) y \cdot \eta ,
\end{equation*}
near $( 0 , 0 )$. Here, $\SB$ is an $n \times n$ smooth matrix-valued function with $\SB ( 0 , 0 ) = \diag ( \lambda_{j} )$. In particular, the outgoing manifold writes $\Lambda_{+} = \{ ( y , 0 ) ; \ y \in \R^{n} \}$. We can suppose that $\SA_{+}$ is sufficiently close to $\Lambda_{+}$, so that the set $\pi_{y} ( \SA_{+} ) = \{ y ; \ ( y , \eta ) \in \SA_{+} \text{ for some } \eta \}$ avoids a neighborhood of $0$. We then consider $\chi_{2} \in C^{\infty}_{0} ( \R ; [ 0 , 1 ] )$ such that $\chi_{2} = 0$ near $0$, $\chi_{2} ( y^{2} ) = 1$ for all $y \in \pi_{y} ( \SA_{+} )$ and $\chi_{2}^{\prime} ( y^{2} ) \geq 0$ for all $y \in \pi_{y} ( \supp \psi_{2} )$. We define
\begin{equation} \label{a16}
g ( x , \xi ) = y^{2} \chi_{2} ( y^{2} ) \psi_{2} ( x , \xi ) \in C^{\infty}_{0} ( \R^{2 n} ) .
\end{equation}
Roughly speaking, $g$ is an escape function in the outgoing region (see \eqref{a2}). Compared to the proofs trying to make the operator elliptic (see e.g. \cite[Section 4]{BoFuRaZe07_01}), this function plays a slightly different role here (see \eqref{a19}). In particular, we do not need an escape function in the incoming region, nor near $0$. This explains why we can take $g = 0$ in these regions.

By construction, there exists a constant $c > 0$ such that
\begin{equation} \label{a26}
\forall \rho \in \SA_{+} , \qquad g ( \rho ) \geq c .
\end{equation}
Moreover, a direct computation gives
\begin{align*}
\{ p_{0}, g \} &= \{ \SB y \cdot \eta , y^{2} \chi_{2} ( y^{2} ) \}   \\
&= \big( \SB y + ( \partial_{\eta} \SB ) y \cdot \eta \big) \cdot \big( 2 y \chi_{2} ( y^{2} ) + 2 y y^{2} \chi_{2}^{\prime} ( y^{2} ) \big)  \\
&= \chi_{2} ( y^{2} ) \big( 2 \SB y \cdot y + \CO ( \eta y^{2} ) \big) + y^{2} \chi_{2}^{\prime} ( y^{2} ) \big( 2 \SB y \cdot y + \CO ( \eta y^{2} ) \big) ,
\end{align*}
on the support of $\psi_{1}$. Recall that $\chi_{2}^{\prime} ( y^{2} ) \geq 0$ for $( y , \eta )$ on this set. Thus, if the previous objects have been constructed with support sufficiently close to $( 0 , 0 )$, we have
\begin{equation} \label{a2}
\{ p_{0}, g \} \geq 0 ,
\end{equation}
on the support of $\psi_{1}$.

For all $t > 0$ fixed, there exists $C > 0$ such that $e^{\pm t \vert \ln h \vert g} \in S ( h^{- C} )$. From the semiclassical pseudodifferential calculus, we get
\begin{align*}
\Op ( e^{- t \vert \ln h \vert g} ) \Op ( e^{t \vert \ln h \vert g} ) &= \Op \Big( 1 + \frac{h}{2 i} t^{2} \vert \ln h \vert^{2} \{ - g , g \} + S \big( h^{2} \vert \ln h \vert^{4} \big) \Big)   \\
&= 1 + \Psi \big( h^{2} \vert \ln h \vert^{4} \big) .
\end{align*}
Using a similar equation for $\Op ( e^{t \vert \ln h \vert g} ) \Op ( e^{- t \vert \ln h \vert g} )$ and the Beals lemma (see \cite[Section 8]{DiSj99_01}), we deduce that $\Op ( e^{- t \vert \ln h \vert g} )$ is invertible for $h$ small enough and
\begin{equation} \label{a3}
\Op ( e^{- t \vert \ln h \vert g} )^{- 1} = \Op ( e^{t \vert \ln h \vert g} ) \big( 1 + \Psi ( h^{3 / 2} ) \big) .
\end{equation}

We define the operator
\begin{equation} \label{a4}
Q = \Op ( e^{- t \vert \ln h \vert g} ) P \Op ( e^{- t \vert \ln h \vert g} )^{- 1} - i \sqrt{h} \Op ( 1 - \psi_{1} ) .
\end{equation}
From the previous discussion, it is well-defined and bounded for $h$ small enough. As in \cite[(4.12)--(4.17)]{BoFuRaZe07_01}, the symbolic calculus in $S ( 1 )$ gives
\begin{equation*}
\Op ( e^{- t \vert \ln h \vert g} ) P \Op ( e^{t \vert \ln h \vert g} ) = P - i t h \vert \ln h \vert \Op ( \{ p_{0} , g \} ) + \Psi ( h^{3 / 2} ) .
\end{equation*}
Then, \eqref{a3} implies
\begin{equation} \label{a5}
Q = P - i t h \vert \ln h \vert \Op ( \{ p_{0} , g \} ) - i \sqrt{h} \Op ( 1 - \psi_{1} ) + \Psi ( h^{3 / 2} ) .
\end{equation}

For $j = 1 , \ldots , n$, we define the annihilation operator
\begin{equation} \label{a6}
A_{j} = \Op ( a_{j} ) = \Op \big( \big( \xi_{j} - \partial_{x_{j}} \varphi_{+} ( x ) \big) \psi_{3} ( x , \xi ) \big) ,
\end{equation}
where $\psi_{3} \in C^{\infty}_{0} ( \R^{2 n} ; [ 0 , 1 ] )$ satisfies $\one_{\Omega_{0}} \prec \psi_{3} \prec \psi_{1}$.

\begin{lemma}\sl \label{a10}
For any finite sequence $\alpha = ( \alpha_{1} , \ldots , \alpha_{\vert \alpha \vert} )$ of elements in $\{ 1 , \ldots , n \}$ and of length $\vert \alpha \vert$, we have
\begin{equation*}
A^{\alpha} Q = ( Q - i h \alpha ( \lambda ) ) A^{\alpha} + \sum_{\vert \beta \vert = \vert \alpha \vert} h R_{\alpha , \beta} A^{\beta} + h M + \sum_{\vert \gamma \vert < \vert \alpha \vert} \Psi \big( h^{\vert \alpha \vert - \vert \gamma \vert + 1} \vert \ln h \vert \big) A^{\gamma} ,
\end{equation*}
where $A^{\alpha} = A_{\alpha_{1}} \cdots A_{\alpha_{\vert \alpha \vert}}$ and $\alpha ( \lambda ) = \lambda_{\alpha_{1}} + \cdots + \lambda_{\alpha_{\vert \alpha \vert}}$. Eventually, $R_{\alpha , \beta}$ (resp. $M$) is an operator in $\Psi ( 1 )$ whose symbol vanishes at $( 0 , 0 )$ (resp. is supported inside the support of $\nabla \psi_{3} ( x , \xi )$).
\end{lemma}

\begin{proof}
From \eqref{a27}, the change of variables
\begin{equation*}
( x , \xi ) \longmapsto \big( y = \xi - \nabla \varphi_{-} ( x ) , \eta = \xi - \nabla \varphi_{+} ( x ) \big) ,
\end{equation*}
is a local diffeomorphism near $( 0 , 0 )$ (not necessarily symplectic). Since $p_{0}$ vanishes on $\Lambda_{\pm}$, we have $p_{0} ( y , 0 ) = p_{0} ( 0 , \eta ) = 0$ for all $y , \eta$ in this new system of coordinates. In particular, $\partial_{y} p_{0} ( y , 0 ) = \partial_{\eta} p_{0} ( 0 , \eta ) = 0$. Then, applying two times the Taylor formula yields $p_{0} ( y , \eta ) = B ( y , \eta ) y \cdot \eta$ for some $n \times n$ smooth matrix-valued function $B$. Coming back to the original variables, we have
\begin{equation} \label{a7}
p_{0} ( x , \xi ) = B ( x , \xi ) \big( \xi - \nabla \varphi_{-} ( x ) \big) \cdot \big( \xi - \nabla \varphi_{+} ( x ) \big) .
\end{equation}
Eventually, \ref{h5} and \eqref{a27} imply $B ( 0 , 0 ) = I d$.

Combining \eqref{a5}, \eqref{a6} and $\psi_{3} \prec \psi_{1}$, we deduce
\begin{align}
A_{j} Q &= Q A_{j} + [ A_{j} , Q ] \nonumber    \\
&= Q A_{j} + [ A_{j} , P ] - i t h \vert \ln h \vert \big[ A_{j} , \Op ( \{ p_{0} , g \} ) \big] + \Psi ( h^{5 / 2} ) \nonumber  \\
&= Q A_{j} + \big[ \Op \big( ( \xi_{j} - \partial_{x_{j}} \varphi_{+} ) \psi_{3} \big) , \Op ( p_{0} ) \big] + \Psi \big( h^{2} \vert \ln h \vert \big) \nonumber   \\
&= Q A_{j} + i h \Op \big( \big\{ p_{0} , ( \xi_{j} - \partial_{x_{j}} \varphi_{+} ) \psi_{3} \big\} \big) + \Psi \big( h^{2} \vert \ln h \vert \big) .  \label{a8}
\end{align}
Moreover, \eqref{a7} and a direct computation show that
\begin{align*}
\big\{ p_{0} , ( \xi_{j} - \partial_{x_{j}} \varphi_{+} ) \psi_{3} \big\} ={}& \big\{ B ( \xi - \nabla \varphi_{-} ) \cdot ( \xi - \nabla \varphi_{+} ) , ( \xi_{j} - \partial_{x_{j}} \varphi_{+} ) \psi_{3} \big\}   \\
={}& \big\{ B ( \xi - \nabla \varphi_{-} ) \cdot ( \xi - \nabla \varphi_{+} ) , \xi_{j} - \partial_{x_{j}} \varphi_{+} \big\} \psi_{3}  \\
&+ \big\{ B ( \xi - \nabla \varphi_{-} ) \cdot ( \xi - \nabla \varphi_{+} ) , \psi_{3} \big\} ( \xi_{j} - \partial_{x_{j}} \varphi_{+} )  \\
={}& \sum_{k} \big\{ B ( \xi - \nabla \varphi_{-} ) , \xi_{j} - \partial_{x_{j}} \varphi_{+} \big\}_{k} ( \xi_{k} - \partial_{x_{k}} \varphi_{+} ) \psi_{3} + m  \\
={}& \sum_{k , \ell} B_{k , \ell} \big\{ \xi_{\ell} - \partial_{x_{\ell}} \varphi_{-} , \xi_{j} - \partial_{x_{j}} \varphi_{+} \big\} a_{k}  \\
&+ \sum_{k , \ell} \big\{ B_{k , \ell} , \xi_{j} - \partial_{x_{j}} \varphi_{+} \big\} ( \xi_{\ell} - \partial_{x_{\ell}} \varphi_{-} ) a_{k} + m  \\
={}& - \lambda_{j} a_{j} + \sum_{k} r a_{k} + m ,
\end{align*}
since $\{ \xi_{k} - \partial_{x_{k}} \varphi_{+} , \xi_{j} - \partial_{x_{j}} \varphi_{+} \} = \partial^{2}_{x_{j} , x_{k}} \varphi_{+} - \partial^{2}_{x_{k} , x_{j}} \varphi_{+} = 0$, $B_{k , \ell} = \delta_{k , \ell} + r$ and
\begin{equation*}
\big\{ \xi_{\ell} - \partial_{x_{\ell}} \varphi_{-} , \xi_{j} - \partial_{x_{j}} \varphi_{+} \big\} = - \lambda_{j} \delta_{j , \ell} + r ,
\end{equation*}
from \eqref{a27}. In the previous equations and in the following, $r ( x , \xi )$ (resp. $m ( x , \xi )$) denotes a symbol in $S ( 1 )$ with $r ( x , \xi ) = \CO ( x , \xi )$ (resp. supported inside the support of $\nabla \psi_{3} ( x , \xi )$) changing from occurrence to occurrence. Then, \eqref{a8} becomes
\begin{equation} \label{a9}
A_{j} Q = ( Q - i \lambda_{j} h ) A_{j} + \sum_{k} h R A_{k} + h M + \Psi \big( h^{2} \vert \ln h \vert \big) ,
\end{equation}
where $R = \Op ( r )$ and $M = \Op ( m )$. Applying two times this formula, we get
\begin{align*}
A_{j_{1}} A_{j_{2}} Q ={}& A_{j_{1}} ( Q - i \lambda_{j_{2}} h ) A_{j_{2}} + \sum_{k} h A_{j_{1}} R A_{k} + h A_{j_{1}} M + A_{j_{1}} \Psi \big( h^{2} \vert \ln h \vert \big)   \\
={}& \big( Q - i ( \lambda_{j_{1}} + \lambda_{j_{2}} ) h \big)  A_{j_{1}} A_{j_{2}} + \sum_{k} h R A_{k} A_{j_{2}} + h M A_{j_{2}} + \Psi \big( h^{2} \vert \ln h \vert \big) A_{j_{2}}  \\
&+ \sum_{k} h R A_{j_{1}} A_{k} + \sum_{k} \Psi ( h^{2} ) A_{k} + h M + \Psi \big( h^{2} \vert \ln h \vert \big) A_{j_{1}} + \Psi \big( h^{3} \vert \ln h \vert \big)   \\
={}& \big( Q - i ( \lambda_{j_{1}} + \lambda_{j_{2}} ) h \big) A_{j_{1}} A_{j_{2}} + \sum_{k_{1} , k_{2}} h R A_{k_{1}} A_{k_{2}}   \\
&+ h M + \sum_{k_{1}} \Psi \big( h^{2} \vert \ln h \vert \big) A_{k_{1}} + \Psi \big( h^{3} \vert \ln h \vert \big) .
\end{align*}
Iterating this process, we obtain the lemma.
\end{proof}

For $d \in \N^{*}$, let $Q^{d}$ be the $n^{d} \times n^{d}$ matrix of operators whose coefficient $( \alpha , \beta )$, with $\alpha , \beta$ of length $d$, given by
\begin{equation} \label{a12}
Q^{d}_{\alpha , \beta} = ( Q - i h \alpha ( \lambda ) ) \delta_{\alpha , \beta} + h R_{\alpha , \beta} .
\end{equation}
In particular, $Q^{d} \in \Psi ( 1 )$. For $d$ large enough, it verifies the following resolvent estimate.

\begin{proposition}\sl \label{a11}
Let $C > 0$ and $d > ( C + 3 + \im p_{1} ( 0 , 0 ) ) / \lambda_{1}$. For $h$ small enough and $z \in B ( 0 , C h )$, the operator $Q^{d} - z$ is invertible and
\begin{equation*}
\big\Vert ( Q^{d} - z )^{-1} \big\Vert \lesssim h^{- 1} .
\end{equation*}
\end{proposition}

\begin{proof}
Let
\begin{equation}
\widetilde{g} ( x , \xi ) = x \cdot \xi \psi_{4} ( x , \xi ) ,
\end{equation}
with $\psi_{4} \in C^{\infty}_{0} ( \R^{2 n} ; [ 0 , 1 ] )$ with $\psi_{1} \prec \psi_{4}$. Then, we have
\begin{align*}
\{ p_{0} , \widetilde{g} \} &= \{ p_{0} , x \cdot \xi \} \psi_{4} ( x , \xi ) + \{ p_{0} , \psi_{4} ( x , \xi ) \} x \cdot \xi  \\
&= \Big( 2 \xi^{2} + \sum_{j = 1}^{n} \frac{\lambda_{j}^{2}}{2} x_{j}^{2} + \CO \big( ( x , \xi )^{3} \big) \Big) \psi_{4} ( x , \xi ) + \varphi_{4} ( x , \xi ) ,
\end{align*}
with $\varphi_{4} \in C^{\infty}_{0} ( \R^{2 n} )$ supported inside $\supp \nabla \psi_{4}$. In particular,
\begin{equation} \label{a13}
\{ p_{0} , \widetilde{g} \} \geq \mu ( \xi^{2} + x^{2} ) \psi_{4} + \varphi_{4} ,
\end{equation}
for some $\mu > 0$ if $\supp \psi_{4}$ is sufficiently close to $( 0 , 0 )$.

For $z \in B ( 0 , C h )$ and $s > 0$, we define
\begin{equation} \label{a14}
\widetilde{Q} = \big( \Op ( e^{- s \widetilde{g}} ) \otimes I d \big) \circ ( Q^{d} - z ) \circ \big( \Op ( e^{s \widetilde{g}} ) \otimes I d \big) .
\end{equation}
From \eqref{a5}, \eqref{a12} and $e^{\pm s \widetilde{g}} \in S ( 1 )$, the pseudodifferential calculus gives
\begin{align*}
\widetilde{Q}_{\alpha , \beta} &= \Op ( e^{- s \widetilde{g}} ) \big( ( Q - i h \alpha ( \lambda ) - z ) \delta_{\alpha , \beta} + h R_{\alpha , \beta} \big) \Op ( e^{s \widetilde{g}} )   \\
&= ( Q - i h \alpha ( \lambda ) - z ) \delta_{\alpha , \beta} + h R_{\alpha , \beta} + \Op ( e^{- s \widetilde{g}} ) \big[ \Op ( p_{0} ) , \Op ( e^{s \widetilde{g}} ) \big] \delta_{\alpha , \beta} + \Psi ( h^{3 / 2} )  \\
&= ( Q - i h \alpha ( \lambda ) - z ) \delta_{\alpha , \beta} + h R_{\alpha , \beta} - i h \Op ( e^{- s \widetilde{g}} ) \Op \big( \{ p_{0} , e^{s \widetilde{g}} \} \big) \delta_{\alpha , \beta} + \Psi ( h^{3 / 2} )  \\
&= ( Q - i h \alpha ( \lambda ) - z ) \delta_{\alpha , \beta} + h R_{\alpha , \beta} - i s h \Op ( \{ p_{0} , \widetilde{g} \} ) \delta_{\alpha , \beta} + \Psi ( h^{3 / 2} ) .
\end{align*}
since
\begin{equation} \label{a72}
\Op ( e^{- s \widetilde{g}} ) \Op ( e^{s \widetilde{g}} ) = 1 + \frac{h}{2 i} \Op \big( \{ e^{- s \widetilde{g}} , e^{s \widetilde{g}} \} \big) + \Psi ( h^{2} ) = 1 + \Psi ( h^{2} ) .
\end{equation}
Taking the imaginary part of the previous equation, \eqref{a5} yields
\begin{align}
- \im \widetilde{Q} &= - \frac{\widetilde{Q} + \widetilde{Q}^{*}}{2 i}    \nonumber \\
&= \diag \big( - h \Op ( \im p_{1} ) + t h \vert \ln h \vert \Op ( \{ p_{0} , g \} ) + \sqrt{h} \Op ( 1 - \psi_{1} )  \nonumber \\
&\qquad \qquad \qquad \quad + s h \Op ( \{ p_{0} , \widetilde{g} \} ) + h \alpha ( \lambda ) + \im z \big) + h R + \Psi ( h^{3 / 2} )   \nonumber \\
&\geq \diag \big( - h \im p_{1} ( 0 , 0 ) + t h \vert \ln h \vert \Op ( \{ p_{0} , g \} ) + \sqrt{h} \Op ( 1 - \psi_{1} )  \nonumber \\
&\qquad \qquad \qquad \quad + s h \Op ( \{ p_{0} , \widetilde{g} \} ) + h \lambda_{1} d + \im z \big) + h R + \Psi ( h^{3 / 2} ) .
\end{align}
Note that $- \im \widetilde{Q} \in \Psi ( h^{1 / 2} )$. From \eqref{a2}, \eqref{a13}, $\varphi_{4} \prec 1 - \psi_{1}$, that the symbol of $R$ vanishes at $( 0 , 0 )$ and the assumptions of Proposition \ref{a11}, its symbol satisfies
\begin{equation*}
- \im \widetilde{q} ( x , \xi ) \geq \left\{ \begin{aligned}
&\sqrt{h} - \CO_{s} ( h \vert \ln h \vert ) &&\text{for } x \notin \supp \psi_{1} ,  \\
&s \mu \nu^{2} h - \CO ( h ) - \CO_{s} ( h^{3 / 2} ) &&\text{for } x \in \supp \psi_{1} \setminus B ( 0 , \nu ) ,  \\
&3 h - \CO ( \nu h ) - \CO_{s} ( h^{3 / 2} ) \qquad &&\text{for } x \in B ( 0 , \nu ) ,
\end{aligned} \right.
\end{equation*}
as an $n^{d} \times n^{d}$ matrix, uniformly for $\nu > 0$ sufficiently small. In the previous equation, $\CO_{s} ( 1 )$ denotes a function which is bounded by a constant which may depend on the parameter $s$. Taking $\nu > 0$ small enough and then $s$ large enough, we obtain in the sense of $n^{d} \times n^{d}$ matrices
\begin{equation}
- \im \widetilde{q} ( x , \xi ) \geq 2 h ,
\end{equation}
for all $( x , \xi ) \in \R^{2 n}$ and $h$ small enough. Therefore, G{\aa}rding's inequality implies
\begin{equation*}
- \im \widetilde{Q} \geq 2 h - \CO ( h^{3 / 2} ) .
\end{equation*}
In other words,
\begin{equation*}
\big\Vert \widetilde{Q} u \big\Vert \Vert u \Vert \geq - \im \big( \widetilde{Q} u , u \big) \geq \big( 2 h - \CO ( h^{3 / 2} ) \big) \Vert u \Vert^{2} \geq h \Vert u \Vert^{2} ,
\end{equation*}
for all $u \in L^{2} ( \R^{n} )$ and $h$ small enough. Thus, $\Vert \widetilde{Q} u \Vert \geq h \Vert u \Vert$. Since the adjoint of $\widetilde{Q}$ satisfies a similar estimate, $\widetilde{Q}$ is invertible and
\begin{equation} \label{a15}
\big\Vert \widetilde{Q}^{- 1} \big\Vert \leq h^{- 1} .
\end{equation}

Finally, using $e^{\pm s \widetilde{g}} \in S ( 1 )$ and \eqref{a72}, Calder\'{o}n--Vaillancourt's theorem shows that the operators $\Op ( e^{\pm s \widetilde{g}} )$ are invertible with uniformly bounded inverse for $h$ small enough. Combining with \eqref{a14} and \eqref{a15}, the proposition follows.
\end{proof}

Let $u$ be a function in $L^{2} ( \R^{n} )$ satisfying the assumptions of Theorem \ref{a1}. We define
\begin{equation} \label{a21}
v = \Op ( e^{- t \vert \ln h \vert g} ) \Op ( \psi_{5} ) u ,
\end{equation}
with $\psi_{5} \in C^{\infty}_{0} ( \R^{2 n} ; [ 0 , 1 ] )$ such that $\one_{\Omega_{1}} \prec \psi_{5} \prec \one_{\Omega_{0}}$. In particular, the derivatives of $\psi_{5}$ are supported inside $\Omega_{0} \setminus \Omega_{1}$. The function $v \in \CS ( \R^{n} )$ satisfies the following estimates.

\begin{lemma}\sl \label{a24}
Let $d_{0} = [ ( C + 3 + \im p_{1} ( 0 , 0 ) ) / \lambda_{1} ] + 1 \in \N$ and $t > 0$. We have
\begin{equation} \label{a73}
\Vert A^{\alpha} v \Vert_{L^{2} ( \R^{n} )} = \CO \big( h^{\vert \alpha \vert - d_{0}} \vert \ln h \vert^{\vert \alpha \vert - d_{0}} \big) ,
\end{equation}
for all finite sequences $\alpha$ satisfying $0 \leq \vert \alpha \vert \leq c t + d_{0} - 1$.
\end{lemma}

\begin{proof}
For the first values of $\vert \alpha \vert$, we use $A_{\bullet} \in \Psi ( 1 )$, $\Vert u \Vert \leq 1$ and $e^{- t \vert \ln h \vert g} \in S ( h^{- 1 / 2} )$ since $g \geq 0$ (see \eqref{a16}). Combining with \eqref{a21}, it implies
\begin{equation} \label{a25}
A^{\alpha} v = A^{\alpha} \Op ( e^{- t \vert \ln h \vert g} ) \Op ( \psi_{5} ) u = \CO ( h^{- 1 / 2} ) = \CO \big( h^{\vert \alpha \vert - d_{0}} \vert \ln h \vert^{\vert \alpha \vert - d_{0}} \big) ,
\end{equation}
for all finite sequence $\alpha$ with $\vert \alpha \vert < d_{0}$.

The proof for the larger values of $\vert \alpha \vert$ rests on the propagation equation on $u$. Using that the supports of $1 - \psi_{1}$ and $\psi_{5}$ are disjoint, \eqref{a4} gives
\begin{align}
( Q - z ) v &= \Op ( e^{- t \vert \ln h \vert g} ) ( P - z ) \Op ( \psi_{5} ) u + \CO ( h^{\infty} )  \nonumber \\
&= \Op ( e^{- t \vert \ln h \vert g} ) \Op ( \psi_{5} ) ( P - z ) u + \Op ( e^{- t \vert \ln h \vert g} ) \big[ P , \Op ( \psi_{5} ) \big] u + \CO ( h^{\infty} )   \nonumber \\
&= \Op ( e^{- t \vert \ln h \vert g} ) \big( \Op ( \psi_{\SA_{-}} ) + \Op ( \psi_{\SA_{0}} ) + \Op ( \psi_{\SA_{+}} ) \big) u + \CO ( h^{\infty} ) , \label{a17}
\end{align}
since $( P - z ) u = 0$ microlocally near $\supp \psi_{5}$. In the previous expression, $\psi_{\SA_{\bullet}} \in S ( h )$ is supported inside $\SA_{\bullet}$. By assumption, $u = 0$ microlocally near $\SA_{-}$. Since $\SA_{0}$ is geometrically controlled by $\SA_{-}$, the standard propagation of singularities implies that $u = 0$ microlocally near $\SA_{0}$. In particular,
\begin{equation} \label{a18}
\Op ( \psi_{\SA_{-}} ) u = \CO ( h^{\infty} ) \qquad \text{and} \qquad \Op ( \psi_{\SA_{0}} ) u = \CO ( h^{\infty} ) .
\end{equation}
For the remainder term of \eqref{a17}, we write
\begin{equation*}
\Op ( e^{- t \vert \ln h \vert g} ) \Op ( \psi_{\SA_{+}} ) = \Op ( e^{- t \vert \ln h \vert g} \varphi_{\SA_{+}} ) \Op ( \psi_{\SA_{+}} ) + \CO ( h^{\infty} ) ,
\end{equation*}
where $\varphi_{\SA_{+}} \in C^{\infty}_{0} ( \SA_{+} ; [ 0 , 1 ] )$ is any function such that $\psi_{\SA_{+}} \prec \varphi_{\SA_{+}}$. On the other hand, \eqref{a26} shows that $e^{- t \vert \ln h \vert g} \varphi_{\SA_{+}} \in S ( h^{c t - 1} )$ and then
\begin{equation} \label{a19}
\Op ( e^{- t \vert \ln h \vert g} ) \Op ( \psi_{\SA_{+}} ) u = \CO ( h^{c t} ) .
\end{equation}
Thus, \eqref{a17} together with \eqref{a18} and \eqref{a19} gives
\begin{equation} \label{a20}
( Q - z ) v = \CO ( h^{c t} ) .
\end{equation}

Moreover, since $\psi_{5} = 0$ near the support of the symbol $m$ of $M$, we have $M v = \CO ( h^{\infty} )$. Then, Lemma \ref{a10}, \eqref{a12}, \eqref{a20} and $A_{\bullet} \in \Psi ( 1 )$ lead to
\begin{equation*}
( Q^{d} - z ) ( A^{\alpha} v )_{\vert \alpha \vert = d} = \sum_{\vert \gamma \vert < d} \Psi \big( h^{d - \vert \gamma \vert + 1} \vert \ln h \vert \big) A^{\gamma} v + \CO ( h^{c t} ) ,
\end{equation*}
for all $d \in \N$ where $( A^{\alpha} v )_{\vert \alpha \vert = d}$ is an $n^{d}$ vector of $L^{2} ( \R^{n} )$. From Proposition \ref{a11}, we get
\begin{align}
A^{\alpha} v &= \sum_{\vert \gamma \vert < \vert \alpha \vert} \CO \big( h^{\vert \alpha \vert - \vert \gamma \vert } \vert \ln h \vert \big) A^{\gamma} v + \CO ( h^{c t - 1} )   \nonumber \\
&= \sum_{\vert \gamma \vert < \vert \alpha \vert} \CO \big( h^{\vert \alpha \vert - \vert \gamma \vert} \vert \ln h \vert^{\vert \alpha \vert - \vert \gamma \vert} \big) A^{\gamma} v + \CO ( h^{c t - 1} ) ,  \label{a22}
\end{align}
for all $\alpha$ with $\vert \alpha \vert \geq d_{0}$. We now prove the required result by induction over $\vert \alpha \vert$. Assume that \eqref{a73} holds true for all sequences $\alpha$ with $\vert \alpha \vert < d$ and $d_{0} \leq d \leq c t + d_{0} -1$. Then, \eqref{a22} and the induction hypothesis imply
\begin{align*}
A^{\alpha} v &= \sum_{\vert \gamma \vert < \vert \alpha \vert} \CO \big( h^{\vert \alpha \vert - \vert \gamma \vert} \vert \ln h \vert^{\vert \alpha \vert - \vert \gamma \vert} \big) \CO \big( h^{\vert \gamma \vert - d_{0}} \vert \ln h \vert^{\vert \gamma \vert - d_{0}} \big) + \CO ( h^{c t - 1} )  \\
&= \CO \big( h^{\vert \alpha \vert - d_{0}} \vert \ln h \vert^{\vert \alpha \vert - d_{0}} \big) + \CO ( h^{c t - 1} ) = \CO \big( h^{\vert \alpha \vert - d_{0}} \vert \ln h \vert^{\vert \alpha \vert - d_{0}} \big) ,
\end{align*}
for all $\alpha$ of length $d$. For the last equality, we have used $c t - 1 \geq d - d_{0} \geq 0$. Thus, \eqref{a73} holds true for all sequences $\alpha$ with $\vert \alpha \vert \leq d$, and the lemma follows by induction.
\end{proof}

Let $\psi_{6} \in C^{\infty}_{0} ( \R^{2 n} ; [ 0 , 1 ] )$ be a cut-off function with $\one_{( 0 , 0 )} \prec \psi_{6} \prec \psi_{5}$ and $\psi_{6} = 0$ near the support of $g$. We also consider a base space cut-off function $\chi_{6} \in C^{\infty}_{0} ( \R^{n} ; [ 0 , 1 ] )$ with $\one_{0} \prec \chi_{6} \prec \one_{\pi_{x} ( \Lambda_{+} \cap \supp \psi_{6} )}$. We then define
\begin{equation} \label{a30}
w = \chi_{6} \Op ( \psi_{6} ) u .
\end{equation}
It is important to note that $w$ is independent of $t$, contrarily to $v$. The localization properties of $\psi_{6}$ and \eqref{a21} give
\begin{align*}
w &= \chi_{6} \Op ( \psi_{6} ) \Op ( \psi_{5} ) u + \CO ( h^{\infty} )   \\
&= \chi_{6} \Op ( \psi_{6} ) \Op ( e^{- t \vert \ln h \vert g} ) \Op ( \psi_{5} ) u + \CO ( h^{\infty} )    \\
&= \chi_{6} \Op ( \psi_{6} ) v + \CO ( h^{\infty} ) .
\end{align*}
Commuting the $A_{\bullet}$'s appearing in $A^{\alpha}$ with $\chi_{6} \Op ( \psi_{6} )$, we get
\begin{align}
A^{\alpha} w &= A^{\alpha} \chi_{6} \Op ( \psi_{6} ) v + \CO ( h^{\infty} )   \nonumber \\
&= \chi_{6} \Op ( \psi_{6} ) A^{\alpha} v + \sum_{\vert \beta \vert < \vert \alpha \vert} \CO ( h^{\vert \alpha \vert - \vert \beta \vert} ) A^{\beta} v + \CO ( h^{\infty} ) .
\end{align}
Thus, Lemma \ref{a24} gives $A^{\alpha} w = \CO ( h^{\vert \alpha \vert - d_{0}} \vert \ln h \vert^{\vert \alpha \vert - d_{0}} )$ for all $0 \leq \vert \alpha \vert \leq c t + d_{0} - 1$. Since $w$ is independent of $t$, this parameter can be taken as large as we want. Then, we have
\begin{equation} \label{a29}
A^{\alpha} w = \CO ( h^{\vert \alpha \vert - d_{0} - 1} ) ,
\end{equation}
for all finite sequence $\alpha$. It implies

\begin{lemma}\sl \label{a28}
There exists $a ( x , h ) \in S ( h^{- d_{0} - 1} )$ supported inside $\supp \chi_{6}$ such that
\begin{equation*}
w ( x , h ) = a ( x , h ) e^{i \varphi_{+} ( x ) / h} .
\end{equation*}
\end{lemma}

This result follows from the standard characterization of the Lagrangian distributions in terms of pseudodifferential operators with symbol vanishing on the corresponding Lagrangian manifold (see e.g. H{\"o}rmander \cite[Proposition 25.1.5]{Ho94_01}). For the sake of completeness, we give the details in the present semiclassical setting.

\begin{proof}
We define $a = e^{- i \varphi_{+} / h} w$. From \eqref{a30}, this $C^{\infty}$ function is supported inside $\supp \chi_{6}$. For a finite sequence $\alpha = ( \alpha_{1} , \ldots , \alpha_{\vert \alpha \vert} )$ of elements in $\{ 1 , \ldots , n \}$, we set $D^{\alpha} = D_{x_{\alpha_{1}}} \cdots D_{x_{\alpha_{\vert \alpha \vert}}}$ and we have
\begin{align*}
D^{\alpha} a &= D_{x_{\alpha_{1}}} \cdots D_{x_{\alpha_{\vert \alpha \vert}}} e^{- i \varphi_{+} / h} w    \\
&= h^{- \vert \alpha \vert} e^{- i \varphi_{+} / h} \big( h D_{x_{\alpha_{1}}} - \partial_{x_{\alpha_{1}}} \varphi_{+} \big) \cdots \big( h D_{x_{\alpha_{\vert \alpha \vert}}} - \partial_{x_{\alpha_{\vert \alpha \vert}}} \varphi_{+} \big) w .
\end{align*}
Since $\psi_{6} \prec \psi_{3}$, \eqref{a6} yields
\begin{equation*}
D^{\alpha} a = h^{- \vert \alpha \vert} e^{- i \varphi_{+} / h} A_{\alpha_{1}} \cdots A_{\alpha_{\vert \alpha \vert}} w + \CO ( h^{\infty} ) = h^{- \vert \alpha \vert} e^{- i \varphi_{+} / h} A^{\alpha} w + \CO ( h^{\infty} ) .
\end{equation*}
Thus, \eqref{a29} implies $\Vert D^{\alpha} a \Vert_{L^{2} ( \R^{n} )} = \CO ( h^{- d_{0} - 1} )$. In other words, we have
\begin{equation}
\Vert a \Vert_{H^{s} ( \R^{n} )} = \CO ( h^{- d_{0} - 1} ) ,
\end{equation}
for all $s \in \R$. The usual Sobolev embeddings give $\Vert a \Vert_{C^{k} ( \R^{n} )} = \CO ( h^{- d_{0} - 1} )$ for all $k \in \N$. Then, $a \in S ( h^{- d_{0} - 1} )$ and the lemma follows.
\end{proof}

We now prove that $a ( x , h ) = \CO ( h^{\infty} )$ near $0$. The first step is to demonstrate that the germ of $a$ at $0$ is small. For that, we define on $\R^{n}$ the vector field
\begin{equation*}
H_{p_{0}}^{+} = \partial_{\xi} p_{0} ( x , \nabla \varphi_{+} ( x ) ) \cdot \partial_{x} .
\end{equation*}
It corresponds to the restriction of the Hamiltonian vector field to $\Lambda_{+}$. Then, $ ( x ( t ) , \xi ( t ) )$ is a Hamiltonian trajectory if and only if $x ( t )$ is an integral curve of $H_{p_{0}}^{+}$ and $\xi ( t ) = \nabla \varphi_{+} ( x ( t ) )$. We also define
\begin{equation} \label{a36}
x^{+} ( t , x ) = \exp ( t H_{p_{0}}^{+} ) ( x ) ,
\end{equation}
the associated flow of $H_{p_{0}}^{+}$. In particular, there exists $c_{0} > 0$ such that
\begin{equation} \label{a33}
\vert x^{+} ( t , x ) \vert \leq e^{- c_{0} \vert t \vert} \vert x \vert ,
\end{equation}
for all $t \leq 0$ and $x$ small enough. Let $W_{1} , W_{2} , \ldots$ and $W_{\infty}$ be open subsets of $\R^{n}$ such that
\begin{equation*}
0 \in W_{\infty} \Subset \cdots \Subset W_{2} \Subset W_{1} \Subset \overset{\circ}{\supp \chi_{6}} ,
\end{equation*}
and such that $x \in W_{j}$, $t \leq 0$ imply $x^{+} ( t , x ) \in W_{j}$. The construction of such sets follows from the Hartman--Grobman theorem.

By assumption, we have $( P -z ) u = 0$ microlocally near $\Omega$. Moreover, \eqref{a30} implies that $w = u$ microlocally near each point of the interior of $( \chi_{6} \psi_{6})^{- 1} ( 1 )$. Then, the rules of calculus for a pseudodifferential operator applied to a Lagrangian distribution give
\begin{align}
\partial_{\xi} p_{0} ( x , & \nabla \varphi_{+} ( x ) ) \cdot \nabla a ( x , h )   \nonumber \\
&+ \Big( \frac{1}{2} \divergence_{x} \big( \partial_{\xi} p_{0} ( x , \nabla \varphi_{+} ( x ) ) \big) + i p_{1} ( x , \nabla \varphi_{+} ( x ) ) - i \frac{z}{h} \Big) a ( x , h ) \in S ( h^{- d_{0}} ) ,  \label{a31}
\end{align}
for $x \in W_{1}$. We refer to Theorem 25.2.4 and (25.2.11) of H\"{o}rmander \cite{Ho94_01} in the classical case, the semiclassical one being a straightforward generalization. From \ref{h5} and \eqref{a27}, this equation can be written
\begin{equation} \label{a65}
\big( L x + \SF ( x^{2} ) \big) \cdot \nabla a ( x , h ) + \Big( \sum_{j = 1}^{n} \frac{\lambda_{j}}{2} + i p_{1} ( 0 , 0 ) - i \frac{z}{h} + \SF ( x ) \Big) a ( x , h ) \in S ( h^{- d_{0}} ) ,
\end{equation}
where $L = \diag ( \lambda_{j} )$ and $\SF ( x^{m} )$ denotes a smooth function, which is $\CO ( x^{m} )$ near $0$, changing from occurrence to occurrence. After derivation, this equation becomes
\begin{align*}
\big( L x + \SF ( x^{2} ) \big) \cdot \nabla \partial_{x}^{\alpha} a & ( x, h ) + \Big( \sum_{j = 1}^{n} \lambda_{j} \Big( \frac{1}{2} + \alpha_{j} \Big) + i p_{1} ( 0 , 0 ) - i \frac{z}{h} \Big) \partial_{x}^{\alpha} a ( x , h ) \\
&+ \sum_{\vert \beta \vert = \vert \alpha \vert} \SF ( x ) \partial_{x}^{\beta} a ( x , h ) + \sum_{\vert \beta \vert \leq \vert \alpha \vert - 1} \SF ( 1 ) \partial_{x}^{\beta} a ( x , h ) \in S ( h^{- d_{0}} ) .
\end{align*}
Taking $x = 0$ leads to
\begin{equation}
\Big( \sum_{j = 1}^{n} \lambda_{j} \Big( \frac{1}{2} + \alpha_{j} \Big) + i p_{1} ( 0 , 0 ) - i \frac{z}{h} \Big) \partial_{x}^{\alpha} a ( 0 , h ) = \sum_{\vert \beta \vert \leq \vert \alpha \vert - 1} \CO \big( \partial_{x}^{\beta} a ( 0 , h ) \big) + \CO ( h^{- d_{0}} ) .
\end{equation}
Since $z$ is at distance at least $\delta h$ from $\Gamma_{0} ( h )$ by assumption, a recurrence over $\vert \alpha \vert$ gives

\begin{lemma}\sl \label{a32}
For all multi-index $\alpha \in \N^{n}$, we have $\partial_{x}^{\alpha} a ( 0 , h ) = \CO ( h^{- d_{0}} )$.
\end{lemma}

We now show that $a$ is small in $W_{1}$ seeing \eqref{a31} as a propagation equation. Combining \eqref{a36} and \eqref{a31}, we get
\begin{equation*}
\partial_{t} a ( x^{+} ( t , x ) , h ) = G ( x^{+} ( t , x ) , h ) a ( x^{+} ( t , x ) , h ) + \CO ( h^{- d_{0}} ) ,
\end{equation*}
where
\begin{equation*}
G ( x , h ) = - \frac{1}{2} \divergence_{x} \big( \partial_{\xi} p_{0} ( x , \nabla \varphi_{+} ( x ) ) \big) - i p_{1} ( x , \nabla \varphi_{+} ( x ) ) + i \frac{z}{h} .
\end{equation*}
It implies
\begin{equation*}
\partial_{t} \Big( e^{- \int_{0}^{t} G ( x^{+} ( s , x ) , h ) \, d s} a ( x^{+} ( t , x ) , h ) \Big) = \CO ( h^{- d_{0}} ) e^{- \int_{0}^{t} G ( x^{+} ( s , x ) , h ) \, d s} ,
\end{equation*}
and then
\begin{equation}
a ( x , h ) = e^{- \int_{0}^{t} G ( x^{+} ( s , x ) , h ) \, d s} a ( x^{+} ( t , x ) , h ) + \int_{t}^{0} \CO ( h^{- d_{0}} ) e^{- \int_{0}^{s} G ( x^{+} ( u , x ) , h ) \, d u} d s ,
\end{equation}
for all $x \in W_{1}$ and $t \leq 0$. Note that $x \in W_{1}$ and $t \leq 0$ ensure that $x^{+} ( t , x ) \in W_{1}$. Moreover, there exists $C > 0$ such that $\vert G ( x , h ) \vert \leq C$ for all $x \in W_{1}$ and $h \in ] 0 , 1 ]$. Thus, the previous equation yields
\begin{equation} \label{a34}
\vert a ( x , h ) \vert \leq e^{C \vert t \vert} \vert a ( x^{+} ( t , x ) , h ) \vert + \CO ( h^{- d_{0}} ) e^{C \vert t \vert} ,
\end{equation}
for all $x \in W_{1}$ and $t \leq 0$. On the other hand, the Taylor formula, $a \in S ( h^{- d_{0} - 1} )$ and Lemma \ref{a32} give
\begin{equation*}
a ( y , h ) = \CO_{K} ( h^{- d_{0}} ) + \CO_{K} ( h^{- d_{0} - 1} \vert y \vert^{K} ) ,
\end{equation*}
for all $K \in \N$ and $y \in W_{1}$. Recall that $\CO_{K} ( 1 )$ designs a function which is bounded by a constant which may depend on the parameter $K$. Combining with \eqref{a33} and \eqref{a34}, we deduce
\begin{equation}
\vert a ( x , h ) \vert \lesssim \CO_{K} \big( h^{- d_{0} - 1} e^{C \vert t \vert} e^{- c_{0} K \vert t \vert} + h^{- d_{0}} e^{C \vert t \vert} \big) ,
\end{equation}
for all $t \leq 0$ and $K \in \N$. Taking
\begin{equation*}
K \geq \frac{2 C}{c_{0}} \qquad \text{and} \qquad t = \frac{\ln h}{2 C} ,
\end{equation*}
we obtain $\vert a ( x , h ) \vert \lesssim h^{- d_{0} - 1 / 2}$ for all $x \in W_{1}$. Eventually, the Landau--Kolmogorov inequalities (see e.g. Ditzian \cite{Di89_01}) imply $a \in S ( h^{- d_{0} - 1 + 1 / 4} )$ in $W_{2}$. In other words, starting from $a \in S ( h^{- d_{0} - 1} )$ in $W_{1}$, we have just proved that $a \in S ( h^{- d_{0} - 1 + 1 / 4} )$ in $W_{2}$. By iteration, we deduce $a \in S ( h^{- d_{0} - 1 + m / 4} )$ in $W_{m + 1}$. Then,

\begin{lemma}\sl \label{a35}
We have $a \in S ( h^{\infty} )$ in $W_{\infty}$.
\end{lemma}

Let $V$ be a closed neighborhood of $( 0 , 0 )$ such that $\pi_{x} ( V ) \subset W_{\infty}$ and $\chi_{6} ( x ) \psi_{6} ( x , \xi ) = 1$ near $V$. From \eqref{a30} and Lemma \ref{a28}, we have $u = a e^{i \varphi_{+} / h}$ microlocally near $V$. Eventually, Lemma \ref{a35} shows that $u = 0$ microlocally near $V$. This ends the proof of Theorem \ref{a1}.

\section{Proof of the asymptotic of resonances} \label{s4}

In order to prove Theorem \ref{a37}, we follow the general strategy to get the spectral asymptotic explained in Section 1.2 of \cite{BoFuRaZe16_01}. First, we show that $P$ has no resonance away from the pseudo-resonances. Then, we prove that $P$ has a resonance close to each pseudo-resonance.

\begin{lemma}\sl \label{a39}
Let $C , \varepsilon > 0$. For $h$ small enough, $P$ has no resonance in $B ( E_{0} , C h ) \setminus ( \res_{0} ( P ) + B ( 0 , \varepsilon h ) )$. Moreover, there exists $N > 0$ such that
\begin{equation*}
\big\Vert( P_{\theta} -z )^{-1} \big\Vert \leq h^{- N} ,
\end{equation*}
uniformly for $h$ small enough and $z \in B ( E_{0} , C h ) \setminus ( \res_{0} ( P ) + B ( 0 , \varepsilon h ) )$.
\end{lemma}

\begin{proof}
From Section 8 of \cite{BoFuRaZe16_01}, it is enough to show that any polynomially bounded solution $u$ of the equation
\begin{equation*}
\left\{ \begin{aligned}
&( P - z ) u = 0 &&\text{ microlocally near } K ( E_{0} ) ,   \\
&u = 0 &&\text{ microlocally in the incoming region} ,
\end{aligned} \right.
\end{equation*}
vanishes microlocally near $K ( E_{0} )$. Thus, this lemma is a direct consequence of the propagation of singularities of Theorem \ref{a1}.
\end{proof}

We now show that there exists at least one resonance near each pseudo-resonance. This is the aim of the following lemma which, combined with Lemma \ref{a39}, implies Theorem \ref{a37}. Let $( \mu_{k} )_{k \in \N}$ be the strictly increasing sequence of linear combinations over $\N$ of the $\lambda_j$'s. In particular, $\mu_{0} = 0$, $\mu_{1} = \lambda_{1}$ and $\mu_{k} \to + \infty$ as $k \to + \infty$.

\begin{lemma}\sl \label{a40}
Let $\nu \in \{ \mu_{k} ; \ k \in \N \}$ and $\varepsilon > 0$. For $h$ small enough, $P$ has at least one resonance in $B ( z_{\nu} , \varepsilon h )$ with $z_{\nu} = E_{0} - i h \nu - i h \sum \lambda_{j} / 2 \in \res_{0} ( P )$.
\end{lemma}

When $\nu$ is simple (i.e. can be written in a unique way as $\nu = \sum \alpha_{j} \lambda_{j}$), this result is a direct consequence of the proof of Theorem 4.1 of \cite{BoFuRaZe11_01} (see more precisely (4.23) and (4.24) of this paper). In the general case, we follow the same strategy: we construct a ``test function'' $v$ and prove that
\begin{equation*}
\oint_{\partial B ( z_{\nu} , \varepsilon h )} ( P_{\theta} - z )^{- 1} v \, d z \neq 0 ,
\end{equation*}
showing that the resolvent $( P_{\theta} - z )^{- 1}$ can not be analytic near $z_{\nu}$. As explained in the introduction of \cite{BoFuRaZe16_01}, this approach can be systematized in spectral theory.

\begin{proof}
We decompose the proof of Lemma \ref{a40} into $5$ steps. We use some intermediate results of Section 6 of \cite{AlBoRa08_01} and of Section 4 and Appendix A of \cite{BoFuRaZe11_01}.

\underline{1st step:} construction of a ``test curve''. Let
\begin{equation*}
Q = \{ ( \alpha_{\mu_{k}} )_{k \geq 1} \in \N^{\N^{*}} ; \ \alpha \cdot \mu = \nu \} \qquad \text{with} \qquad \alpha \cdot \mu = \sum_{k = 1}^{+ \infty} \alpha_{\mu_{k}} \mu_{k} .
\end{equation*}
The set $Q$ is finite and not empty. Thus, we define
\begin{equation} \label{a45}
q_{0} = \max_{\alpha \in Q} \vert \alpha \vert \qquad \text{and} \qquad Q_{\text{max}} = \{ \alpha \in Q ; \ \vert \alpha \vert = q_{0} \} \neq \emptyset ,
\end{equation}
where $\vert \alpha \vert = \alpha_{\mu_{1}} + \alpha_{\mu_{2}} + \cdots$. Note that, if $\alpha_{\mu_{k}} \neq 0$ for some $\alpha \in Q_{\text{max}}$, then $\mu_{k}$ is one of the $\lambda_{j}$'s and can not be decomposed. (i.e. $\mu_{k} = \beta \cdot \lambda$ implies $\vert \beta \vert = 1$). We now fix $\alpha^{0} \in Q_{\text{max}}$ such that $\alpha^{0}_{\mu_{1}} \neq 0$ if it is possible (that is if the set $Q_{\text{max}}$ contains an element $\alpha$ with $\alpha_{\mu_{1}} \neq 0$).

Let $F_{p}$ be the linearization at $( 0 , 0 )$ of the Hamiltonian vector field $H_{p}$, given by
\begin{equation*}
F_{p} =
\left( \begin{array}{cc}
0 & 2 I d \\
\frac{1}{2} \diag ( \lambda_{1}^{2} , \ldots , \lambda_{n}^{2} ) & 0
\end{array} \right) .
\end{equation*}
It is diagonalizable and has eigenvalues $- \lambda_{n} , \dots , - \lambda_{1} , \lambda_{1} , \dots , \lambda_{n}$. We denote by $\Pi_{\lambda}$ the spectral projection on the eigenspace of $F_{p}$ associated to $- \lambda$. From Section 3 of Helffer and Sj\"{o}strand \cite{HeSj85_01}, any Hamiltonian trajectory $\gamma ( t )$ in $\Lambda_{-}$ verifies
\begin{equation} \label{a41}
\gamma ( t ) \simeq \sum_{k = 1}^{+ \infty} \gamma_{\mu_{k}}^{-} ( t ) e^{- \mu_{k} t} \qquad \text{with} \qquad \gamma_{\mu_{k}}^{-} ( t ) = \sum_{m = 0}^{M_{\mu_{k}}} \gamma_{\mu_{k} , m}^{-} t^{m} ,
\end{equation}
as $t \to + \infty$ in the sense of expandible functions. See \cite{HeSj85_01} for the precise definition of expandible, which roughly speaking means that one can derivate term by term these asymptotic expansions. Furthermore, Lemma A.1 of \cite{BoFuRaZe11_01} shows that, if $\lambda_{j}$ can not be decomposed, we have $M_{\lambda_{j}} = 0$ and $\gamma_{\lambda_{j} , 0}^{-} \in \ker ( F_{p} + \lambda_{j} )$. For all $\lambda_{j}$, we choose $\widetilde{\gamma}_{\lambda_{j} , 0}^{-} \in \ker ( F_{p} + \lambda_{j} )$ with
\begin{equation} \label{a46}
\widetilde{\gamma}_{\lambda_{j} , 0}^{-} \left\{ \begin{aligned}
&\neq 0 &&\text{ if } \alpha_{\lambda_{j}}^{0} \neq 0 \text{ or } \lambda_{j} = \lambda_{1} ,  \\
&= 0 &&\text{ otherwise} .
\end{aligned} \right.
\end{equation}
From Proposition A.3 of \cite{BoFuRaZe11_01}, there exists a Hamiltonian trajectory $\gamma ( t ) = ( x ( t ) , \xi ( t ) ) \in \Lambda_{-}$ such that
\begin{equation*}
\forall \lambda_{j} , \qquad \Pi_{\lambda_{j}} ( \gamma^{-}_{\lambda_{j},0} ) = \widetilde{\gamma}^{-}_{\lambda_{j},0} .
\end{equation*}

\underline{2nd step:} construction of a ``test function''. We apply the WKB construction made in Section 4.1 of \cite{BoFuRaZe11_01}. No change has to be made in the present situation. Then, there exists a function
\begin{equation*}
u ( x , z , h ) = b ( x , z , h ) e^{i \psi ( x ) / h} ,
\end{equation*}
satisfying the following properties. First, the smooth phase function $\psi$ solves the eikonal equation $\vert \nabla \psi \vert^{2} + V ( x ) = E_{0}$ and the manifold $\Lambda_{\psi} = \{ ( x , \nabla \psi ( x ) ) \}$ intersects transversally $\Lambda_{-}$ along $\gamma$. The symbol $b \in S ( 1 )$ is supported near the $x$-space projection of $\gamma$ and has an asymptotic expansion
\begin{equation*}
b ( x , z , h ) \simeq \sum_{j = 0}^{\infty} b_{j} ( x , z ) h^{j} ,
\end{equation*}
uniformly for $z \in B ( z_{\nu} , 2 \varepsilon h )$. Furthermore, $b$ and the $b_{j}$'s are holomorphic for $z \in B ( z_{\nu} , 2 \varepsilon h )$. For some $T > 0$ large enough, the principal symbol of $b$ at $x ( T )$ is independent of $z$ and non-zero (i.e. $b_{0} ( x ( T ) , z ) = b_{0} ( x ( T ) ) \neq 0$). Eventually, $u$ satisfies
\begin{equation}
( P - z ) u = \CO ( h^{\infty} ) ,
\end{equation}
near $x ( [ T , + \infty [ )$.

Our test function is defined as
\begin{equation}
v = [ P , \chi ] u ,
\end{equation}
where $\chi \in C^{\infty}_{0} ( \R^{n} ; [ 0 , 1 ] )$ with $\chi = 1$ near $0$ and $\supp \chi$ sufficiently close to $0$. In particular, $v$ is holomorphic in $B ( z_{\nu} , 2 \varepsilon h )$. We set
\begin{equation} \label{a52}
w = ( P_{\theta} - z )^{- 1} v ,
\end{equation}
for $z \in \CR = B ( z_{\nu} , 2 \varepsilon h ) \setminus B ( z_{\nu} , \varepsilon h / 2 )$. From Lemma \ref{a39}, this function is well-defined and holomorphic in $\CR$. Moreover, it satisfies $\Vert w \Vert \lesssim h^{- N}$ uniformly for $z \in \CR$.

\underline{3rd step:} representation of $w$. We first describe $w$ in the incoming region. For $\rho \in \Lambda_{-} \setminus \{ ( 0 , 0 ) \}$ sufficiently close to $( 0 , 0 )$, we have
\begin{equation*}
w = \left\{ \begin{aligned}
&u &&\text{ microlocally near } \rho \text{ if } \rho \in \gamma ,   \\
&0 &&\text{ microlocally near } \rho \text{ if } \rho \notin \gamma .
\end{aligned} \right.
\end{equation*}
This follows from Remark 4.4 and Lemma 4.5 of our previous work \cite{BoFuRaZe11_01}.

We now give a representation formula for $w$ near $( 0 , 0 )$. For that, we follow Section 4.3 of \cite{BoFuRaZe11_01} which rests on Section 5 of \cite{BoFuRaZe07_01}. Let
\begin{equation*}
\sigma = \frac{z - E_{0}}{h} \qquad \text{and} \qquad \sigma_{\nu} = \frac{z_{\nu} - E_{0}}{h} ,
\end{equation*}
be the rescaled spectral parameters. Then, Theorem 5.1 of \cite{BoFuRaZe07_01} shows that
\begin{equation} \label{a42}
w = \frac{1}{\sqrt{2 \pi h}} e^{i \varphi_{+} ( x ) / h} e^{i \psi ( 0 ) / h} A_{-} ( x , \sigma , h ) + \frac{1}{\sqrt{2 \pi h}} \int_{- 1}^{+ \infty} e^{i \varphi ( t , x ) / h} A_{+} ( t , x , \sigma ,h ) \, d t ,
\end{equation}
microlocally near $(0,0)$. The symbol $A_{+}$ is a holomorphic function of $\sigma \in B ( \sigma_{\nu} , 2 \varepsilon )$ which decays exponentially in $t$ uniformly with respect to $x, \sigma ,h$. The constant $\psi ( 0 )$ is defined by
\begin{equation*}
\psi ( 0 ) = \lim_{t \to + \infty} \psi ( x ( t ) ) \in \R .
\end{equation*}

The symbol $A_{-} ( x , \sigma , h ) \in S ( h^{- C} )$ is constructed as follows. First, there is an expandible symbol $a ( t , x , \sigma , h ) \in S ( 1 )$ of the form
\begin{equation*}
a ( t , x , \sigma , h ) \simeq \sum_{j = 0}^{+ \infty} a_{j} ( t , x , \sigma ) h^{j} ,
\end{equation*}
where the $a_{j}$'s satisfy
\begin{equation*}
a_{j} ( t , x , \sigma ) \simeq \sum_{k = 0}^{+ \infty} a_{j , \mu_{k}} ( t , x , \sigma ) e^{- ( S + \mu_{k} ) t} \qquad \text{and} \qquad a_{j , \mu_{k}} ( t , x , \sigma ) = \sum_{\ell = 0}^{M_{j , \mu_{k}}} a_{j , \mu_{k} , \ell} ( x , \sigma ) t^{\ell} .
\end{equation*}
In this expression, $S$ is defined by
\begin{equation*}
S = S ( \sigma ) = \sum_{j = 1}^{n} \frac{\lambda_{j}}{2} - i \sigma .
\end{equation*}
The symbols $a_{j} , a_{j , \mu_{k}} , a_{j , \mu_{k} ,\ell}$ are holomorphic for $\sigma \in B ( \sigma_{\nu} , 2 \varepsilon )$. Finally, $a_{0 , 0}$ does not depend on $t , \sigma$ and
\begin{equation} \label{a50}
a_{0,0} ( 0 ) = \vert g^{-}_{\lambda_{1}} \vert \lambda_{1}^{\frac{3}{2}} e^{- i \frac{\pi}{4}} e^{- \int_{T}^{+ \infty} \Delta \psi (x ( s ) ) - ( \sum \lambda_{j} /2 - \lambda_{1} ) \, d s} b_{0} ( x ( T ) ) ,
\end{equation}
with $g^{-}_{\lambda_{j}} = \pi_{x} ( \gamma_{\lambda_{j} , 0}^{-} )$. On the other hand, there exists an expandible function 
\begin{equation*}
\varphi_{\star} ( t , x ) \simeq \sum_{k = 1}^{+ \infty} \varphi_{\mu_{k}} ( t , x ) e^{- \mu_{k} t} \qquad \text{and} \qquad \varphi_{\mu_{k}} ( t , x ) = \sum_{\ell = 0}^{N_{\mu_{k}}} \varphi_{\mu_{k} , \ell} ( x ) t^{\ell} .
\end{equation*}
With the notations of \cite{BoFuRaZe07_01}, it is defined by $\varphi_{\star} ( t , x ) = \varphi ( t , x ) - ( \varphi_{+} (x) + \psi (0) )$. We consider the expandible symbol
\begin{equation} \label{a44}
\widetilde{a} = \sum_{q < Q_{1}} \frac{a}{q !} \Big( \frac{i \varphi_{\star}}{h} \Big)^{q} \simeq \sum_{j = 1 - Q_{1}}^{+ \infty} \widetilde{a}_{j} ( t , x , \sigma ) h^{j} ,
\end{equation}
for some $Q_{1} \in \N$ fixed large enough,
\begin{equation*}
\widetilde{a}_{j} ( t , x , \sigma ) \simeq \sum_{k = 0}^{+ \infty} a_{j , \mu_{k}} ( t , x , \sigma ) e^{- ( S + \mu_{k} ) t} \qquad \text{and} \qquad \widetilde{a}_{j , \mu_{k}} ( t , x , \sigma ) = \sum_{\ell = 0}^{\widetilde{M}_{j , \mu_{k}}} \widetilde{a}_{j , \mu_{k} , \ell} ( x , \sigma ) t^{\ell} .
\end{equation*}
Then, $A_{-} ( x , \sigma , h )$ is a symbol of class $S ( h^{1 - Q_{1}} )$, holomorphic with respect to $\sigma \in B ( \sigma_{\nu} , 2 \varepsilon ) \setminus B ( \sigma_{\nu} , \varepsilon / 2 )$, such that
\begin{equation}\label{a48}
A_{-} ( x , \sigma , h ) \simeq \sum_{j = 1 - Q_{1}}^{+ \infty} h^{j} \sum_{k = 0}^{K_{1}} \sum_{\ell = 0}^{\widetilde{M}_{j , \mu_{k}}} \frac{\ell !}{( S + \mu_{k})^{\ell + 1}} \widetilde{a}_{j, \mu_{k} , \ell} ( x , \sigma ) ,
\end{equation}
for some $K_{1} \in \N$ large enough.

In the sequel, we will use some additional informations on the $\varphi_{\mu_{k}}$'s. Assume that $\lambda_{j}$ can not be decomposed (see below \eqref{a45}). Working as in Section 4.3 of \cite{BoFuRaZe11_01} and Section 6.1 of \cite{AlBoRa08_01}, one can show that $\varphi_{\lambda_{j}}$ does not depend on $t$ (i.e. $N_{\lambda_{j}} = 0$) and that
\begin{equation} \label{a54}
\varphi_{\lambda_{j}} ( x ) = - \lambda_{j} g_{\lambda_{j}}^{-} \cdot x + \CO ( x^{2} ) .
\end{equation}
If $g_{\lambda_{j}}^{-} = 0$ in addition, then $\varphi_{\lambda_{j}} = 0$.

\underline{4th step:} integration with respect to $z$. We now compute the function
\begin{equation} \label{a53}
f ( x ) = \oint_{\partial B ( z_{\nu} , \varepsilon h )} w ( x , z ) \, d z = h \oint_{\partial B ( \sigma_{\nu} , \varepsilon )} w ( x , \sigma ) \, d \sigma ,
\end{equation}
microlocally near $( 0 , 0 )$. Since $A_{+} ( \sigma )$ is holomorphic in $B ( \sigma_{\nu} , 2 \varepsilon )$, the last term in \eqref{a42} gives no contribution to $f$. Thus, this function can be written
\begin{equation} \label{a43}
f ( x ) = \sqrt{\frac{h}{2 \pi}} e^{i \varphi_{+} ( x ) / h} e^{i \psi ( 0 ) / h} \oint_{\partial B ( \sigma_{\nu} , \varepsilon )} A_{-} ( x , \sigma , h ) \, d \sigma ,
\end{equation}
microlocally near $( 0 , 0 )$. The same way, the terms with $\mu_{k} \neq \nu$ in \eqref{a48} are holomorphic in $B ( \sigma_{\nu} , 2 \varepsilon )$ and do not contribute in $f ( x )$.

So, $f$ is given by the terms with $\mu_{k} = \nu$ in \eqref{a48}. Among these terms, those with the higher possible power of $h^{- 1}$ come from contributions with $q = q_{0}$ in \eqref{a44} and $a_{0} ( t ,x , \sigma ) = a_{0 , 0} ( x )$. Thus, we consider $\alpha \in Q_{\text{max}}$, that is a multi-index such that $\alpha \cdot \mu = \nu$ and $\vert \alpha \vert = q_{0}$. If $\alpha_{\mu_{k}} \neq 0$, then $\mu_{k}$ is one of the $\lambda_{j}$'s and can not be decomposed. Thus, \eqref{a54} implies that $\varphi_{\mu_{k}}$ is independent of $t$. Therefore, the contribution of $\alpha$ in $\widetilde{a}$ (see \eqref{a44}) is given by
\begin{equation} \label{a47}
\frac{a_{0 ,0} ( x )}{\alpha !} \Big( \frac{i}{h} \Big)^{\vert \alpha \vert} \prod_{k \geq 1} \big( \varphi_{\mu_{k}} ( x ) \big)^{\alpha_{\mu_{k}}} .
\end{equation}
If $\alpha \neq \alpha^{0}$, at least one of the $\varphi_{\mu_{k}}$ is null thanks to \eqref{a46} and the line below \eqref{a54}. Then, \eqref{a47} vanishes identically in that case.

Summing up, \eqref{a48} can be written
\begin{equation} \label{a49}
A_{-} ( x , \sigma , h ) \simeq \sum_{j = 0}^{+ \infty} a_{j}^{-} ( x , \sigma ) h^{- q_{0} + j} + \CH ( x , \sigma , h ) ,
\end{equation}
where the $a_{j}^{-}$'s are holomorphic with respect to $\sigma \in B ( \sigma_{\nu} , 2 \varepsilon ) \setminus B ( \sigma_{\nu} , \varepsilon / 2)$ and $C^{\infty}$ with respect to $x$. Moreover, the function $\CH$ is holomorphic in $\sigma \in B ( \sigma_{\nu} , 2 \varepsilon )$. Eventually,
\begin{equation*}
a_{0}^{-} ( x , \sigma ) = \frac{i^{q_{0} + 1}}{\alpha^{0} !} \frac{a_{0 , 0} ( x )}{\sigma - \sigma_{\nu}} \prod_{k \geq 1} \big( \varphi_{\mu_{k}} ( x ) \big)^{\alpha_{\mu_{k}}^{0}} .
\end{equation*}
Thus, \eqref{a43} and \eqref{a49} yield
\begin{equation} \label{a51}
f ( x ) \simeq e^{i \varphi_{+} ( x ) / h} \sum_{j = 0}^{+ \infty} f_{j} ( x ) h^{\frac{1}{2} - q_{0} + j} ,
\end{equation}
microlocally near $( 0 , 0 )$ with 
\begin{equation} \label{a55}
f_{0} ( x ) = - \frac{i^{q_{0}} \sqrt{2 \pi}}{\alpha^{0} !} e^{i \psi ( 0 ) / h} a_{0 , 0} ( x ) \prod_{k \geq 1} \big( \varphi_{\mu_{k}} ( x ) \big)^{\alpha_{\mu_{k}}^{0}} .
\end{equation}

\underline{5th step:} conclusion. The previous paragraph shows that, microlocally near $( 0 , 0 )$, $f ( x )$ is a Lagrangian distribution carried out by $\Lambda_{+}$ and of order $- q_{0} + 1 / 2$. Moreover, its principal symbol $f_{0}$ does not vanishes identically near $0$ thanks to \eqref{a55} together with \eqref{a46}, \eqref{a50} and \eqref{a54}. Therefore, $f$ is not the zero function.

Assume now that $P$ has no resonance in $B ( z_{\nu} , \varepsilon h )$ (and then in $B ( z_{\nu} , 2 \varepsilon h )$). Since $v ( z )$ is holomorphic in $B ( z_{\nu} , 2 \varepsilon h )$, so would be $w ( z )$ from \eqref{a52}. Then, \eqref{a53} would become
\begin{equation*}
f ( x ) = \oint_{\partial B ( z_{\nu} , \varepsilon h )} w ( x , z ) \, d z = 0 ,
\end{equation*}
and we get a contradiction. This proves that $P$ has at least one resonance in $B ( z_{\nu} , \varepsilon h )$ and the lemma follows.
\end{proof}

We end this section by giving the asymptotic modulo $\CO ( h^{\infty} )$ of the resonances near simple pseudo-resonances.

\begin{proof}[Proof of Proposition \ref{a58}]
From Theorem \ref{a37}, we already know that, for $\delta > 0$ small enough, all the resonances $z$ in $B ( z_{0} ( h ) , 2 \delta h )$ satisfy $z = z_{0} ( h ) + o ( h )$ and that there exists at least one such resonance. Then, to obtain the proposition, it is enough to construct $z_{\infty} ( h )$ and to show that, for all $N > 0$, $P$ has no resonance and a polynomial estimate of its truncated resolvent in $B ( z_{\infty} ( h ) , \delta h ) \setminus B ( z_{\infty} ( h ) , h^{N} )$ for $h$ small enough.

For that, it is enough to prove the uniqueness of the microlocal Cauchy problem at the barrier-top (see the proof of Lemma \ref{a39}). This last property is obtained following Section \ref{s3}. Then, we consider a function $u$ satisfying the assumptions of Theorem \ref{a1} with $z \in B ( z_{\infty} ( h ) , \delta h ) \setminus B ( z_{\infty} ( h ) , h^{N} )$. Applying Lemma \ref{a28}, we deduce that
\begin{equation*}
u ( x , h ) = a ( x , h ) e^{i \varphi_{+} ( x ) / h} ,
\end{equation*}
microlocally near $( 0 , 0 )$ for some $a \in S ( h^{- \mu} )$ with $\mu \in \R$. By assumption, there exists a unique $\alpha^{0} \in \N^{n}$ such that
\begin{equation*}
z_{0} ( h ) = E_{0} - i h \sum_{j = 1}^{n} \lambda_{j} \Big( \frac{1}{2} + \alpha_{j}^{0} \Big) .
\end{equation*}
In the sequel, we will also use the notations
\begin{equation*}
\sigma = \frac{z - E_{0}}{h} \qquad \text{and} \qquad \sigma_{0} = \frac{z_{0} ( h ) - E_{0}}{h} = - i \sum_{j = 1}^{n} \lambda_{j} \Big( \frac{1}{2} + \alpha_{j}^{0} \Big) .
\end{equation*}

Since $P$ is a Schr\"{o}dinger operator, the transport equation \eqref{a31} writes
\begin{equation} \label{a59}
2 \nabla \varphi_{+} ( x ) \cdot \nabla a ( x , h ) + \big( \Delta \varphi_{+} ( x ) - i \sigma \big) a ( x , h ) - i h \Delta a ( x , h ) \in S ( h^{\infty} ) ,
\end{equation}
in the present setting. As in \eqref{a65}, this equation can be expressed
\begin{equation*}
\big( L x + \SF ( x^{2} ) \big) \cdot \nabla a ( x , h ) + \Big( \sum_{j = 1}^{n} \frac{\lambda_{j}}{2} - i \sigma + \SF ( x ) \Big) a ( x , h ) - i h \Delta a ( x , h ) \in S ( h^{\infty} ) .
\end{equation*}
Applying $\partial_{x}^{\alpha}$ leads to
\begin{align*}
\big( L x + \SF ( & x^{2} ) \big) \cdot \nabla \partial_{x}^{\alpha} a ( x , h ) + \Big( \sum_{j = 1}^{n} \lambda_{j} \Big( \frac{1}{2} + \alpha_{j} \Big) - i \sigma \Big) \partial_{x}^{\alpha} a ( x , h )  \\
&+ \sum_{\vert \beta \vert = \vert \alpha \vert} \SF ( x ) \partial_{x}^{\beta} a ( x , h ) + \sum_{\vert \beta \vert \leq \vert \alpha \vert - 1} \SF ( 1 ) \partial_{x}^{\beta} a ( x , h ) - i h \Delta \partial_{x}^{\alpha} a ( x , h ) \in S ( h^{\infty} ) .
\end{align*}
Taking $x = 0$ and using the notation $a_{\alpha} = \partial_{x}^{\alpha} a ( 0 , h )$, it becomes
\begin{equation} \label{a60}
\Big( \sum_{j = 1}^{n} \lambda_{j} \Big( \frac{1}{2} + \alpha_{j} \Big) - i \sigma \Big) a_{\alpha} = \sum_{\vert \beta \vert \leq \vert \alpha \vert - 1} c_{\alpha}^{\beta} a_{\beta} + h \sum_{\vert \beta \vert = \vert \alpha \vert + 2} d_{\alpha}^{\beta} a_{\beta} + \CO ( h^{\infty} ) ,
\end{equation}
for some complex numbers $c_{\alpha}^{\beta} , d_{\alpha}^{\beta}$ independent of $h$. For $\alpha \neq \alpha^{0}$, the factor on the left is invertible and then
\begin{equation*}
\forall \alpha \neq \alpha^{0} , \qquad a_{\alpha} = \sum_{\vert \beta \vert \leq \vert \alpha \vert - 1} \SG_{\alpha}^{\beta} ( \sigma ) a_{\beta} + h \sum_{\vert \beta \vert = \vert \alpha \vert + 2} \widetilde{\SG}_{\alpha}^{\beta} ( \sigma ) a_{\beta} + \CO ( h^{\infty} ) ,
\end{equation*}
for some functions $\SG_{\alpha}^{\beta} , \widetilde{\SG}_{\alpha}^{\beta}$ holomorphic in $B ( \sigma_{0} , \nu )$, for some $\nu > 0$, and independent of $h$. Then, a recurrence over $K \in \N$ and then over $\vert \alpha \vert$ implies that
\begin{equation} \label{a64}
\forall \alpha \neq \alpha^{0} , \qquad a_{\alpha} = \sum_{k = 0}^{K} h^{k} \SG_{\alpha , k} ( \sigma ) a_{\alpha^{0}} + \CO ( h^{K + 1 - \mu} ) ,
\end{equation}
where the $\SG_{\alpha , k}$'s are as before and $\SG_{\alpha , 0} = 0$ for $\vert \alpha \vert \leq \vert \alpha^{0} \vert - 1$. In particular, the $\SG_{\alpha , k}$'s are independent of $K$. Finally, inserting these relations in \eqref{a60} with $\alpha = \alpha^{0}$ gives
\begin{equation} \label{a68}
\sum_{k = 0}^{K} h^{k} \SG_{k} ( \sigma ) a_{\alpha^{0}} = \CO ( h^{K + 1 - \mu} ) ,
\end{equation}
where the $\SG_{k}$'s are as before and
\begin{equation} \label{a61}
\SG_{0} ( \sigma ) = \sum_{j = 1}^{n} \lambda_{j} \Big( \frac{1}{2} + \alpha_{j}^{0} \Big) - i \sigma = i ( \sigma_{0} - \sigma ) .
\end{equation}
Using Borel's lemma, one can construct a function $\SG ( \sigma , h )$ holomorphic in $\sigma \in B ( \sigma_{0} , \nu )$ and $C^{\infty}$ in $h \in [ - 1 , 1 ]$ such that $\SG ( \sigma , h ) \simeq \sum_{k \geq 0} h^{k} \SG_{k} ( \sigma )$ as $h \to 0$. Then, \eqref{a68} gives
\begin{equation} \label{a63}
\SG ( \sigma , h ) a_{\alpha^{0}} = \CO ( h^{\infty} ) .
\end{equation}

From \eqref{a61} and the implicit function theorem, there exists $\varepsilon > 0$ and a (unique) $\sigma_{\infty} ( h )$ smooth near $0$ such that
\begin{equation} \label{a62}
\SG ( \sigma_{\infty} ( h ) , h ) = 0 ,
\end{equation}
and $\sigma_{\infty} ( h ) \in B ( \sigma_{0} , \varepsilon )$. In particular, $\sigma_{\infty} ( 0 ) = \sigma_{0}$. We now define $z_{\infty} ( h ) = E_{0} + h \sigma_{\infty} ( h )$. Since $h \mapsto \sigma_{\infty} ( h )$ is smooth, the Taylor formula shows that $z_{\infty} ( h ) \simeq E_{0} + E_{1} h + E_{2} h^{2} + \cdots$ and $E_{1} = - i \sum_{j} \lambda_{j} ( \alpha_{j} + 1 / 2 )$. Using \eqref{a61} and \eqref{a62}, we deduce
\begin{equation*}
\vert \SG ( \sigma , h ) \vert = \big\vert \SG ( \sigma , h ) - \SG ( \sigma_{\infty} ( h ) , h ) \big\vert \gtrsim \big\vert \sigma - \sigma_{\infty} ( h ) \big\vert \geq h^{N - 1} ,
\end{equation*}
uniformly for $z \in B ( z_{\infty} ( h ) , \delta h ) \setminus B ( z_{\infty} ( h ) , h^{N} )$ and $h$ small enough. Therefore, \eqref{a63} gives $a_{\alpha^{0}} = \CO ( h^{\infty} )$, and \eqref{a64} implies $\partial_{x}^{\alpha} a ( 0 , h ) = \CO ( h^{\infty} )$ for all $\alpha \in \N^{n}$. Eventually, the proof of Lemma \ref{a35} yields $a ( x , h ) = \CO ( h^{\infty} )$ near $0$ and Proposition \ref{a58} follows.
\end{proof}

\appendix

\section{Microlocal terminology} \label{s7}

The only aim of this appendix is to fix the terminology we use in the paper. Details can be found for example in the textbooks of  Dimassi and Sj\"ostrand \cite{DiSj99_01}, Martinez \cite{Ma02_02}, Robert \cite{Ro87_01} and Zworski \cite{Zw12_01}.

A function $a (x , h ) \in C^{\infty} ( \R^{d} \times ] 0 , h_{0} ])$ is a symbol in the class $S ( 1 )$ when
\begin{equation*}
\forall \alpha \in \N^{d}, \quad \exists C_{\alpha} > 0 , \quad \forall h \in ] 0 , h_{0} ] , \quad \forall x \in \R^{d} , \qquad \vert \partial^{\alpha}_{x} a ( x , h ) \vert \leq C_{\alpha} .
\end{equation*}
For a positive function $m ( h )$ of $h$ only, we denote $S ( m ( h ) ) = m ( h ) S ( 1 )$. The Weyl quantization of $a ( x , \xi , h ) \in S ( 1 )$ is the operator defined on $\CS^{\prime} ( \R^{n} )$ by
\begin{equation*}
( \Op ( a ) u ) ( x ) = \frac{1}{( 2 \pi h)^{n}} \iint e^{i ( x - y ) \cdot \xi / h} a \Big( \frac{x + y}{2} , \xi , h \Big) u( y ) \, d y \, d \xi .
\end{equation*}
We denote $\Psi ( 1 ) = \Op ( S ( 1 ) )$ the set of pseudodifferential operators with symbol in $S ( 1 )$, and $\Psi ( m ( h ) ) = m ( h ) \Psi ( 1 )$. For $a \in S ( 1 )$, the Calder\'{o}n--Vaillancourt theorem asserts that $\Op ( a )$ is bounded on $L^{2} ( \R^{n} )$ uniformly with respect to $h$.

A function $u ( x , h ) \in L^{2} ( \R^{n} )$ is polynomially bounded if  $\Vert u \Vert = \CO ( h^{- C} )$ for some $C \in \R$. We say that a polynomially bounded function $u$ vanishes microlocally near a bounded set $V \subset T^{*} \R^{n}$ when there exists a plateau function $a ( x , \xi ) \in S ( 1 )$ above $\overline{V}$ such that
\begin{equation*}
\big\Vert \Op ( a ) u \big\Vert = \CO ( h^{\infty} ) .
\end{equation*}
The complement of the set of points where $u$ vanishes microlocally is called the microsupport of $u$, and we denote it by $\MS ( u )$.

A Lagrangian manifold $\Lambda$ of $T^{*} \R^{n}$ is a smooth manifold of dimension $n$ on which the canonical symplectic two-form $\sigma = d \xi \wedge d x$ vanishes. If $\Lambda$ projects nicely on the base space near some point, there exists a smooth function $\varphi$ such that $\Lambda$ can be written $\Lambda = \Lambda_{\varphi} = \{ ( x , \xi ) ; \ \xi = \nabla \varphi ( x ) \}$ near that point. In that case, $\varphi$ is called a generating function of $\Lambda$.

A Lagrangian distribution of order $m ( h )$ associated to a Lagrangian manifold $\Lambda$ is a polynomially bounded function $u$ with $\MS ( u ) \subset \Lambda$ such that, microlocally near each point $\rho \in \Lambda$, $u$ can be written up to a partial Fourier transform as
\begin{equation*}
u ( x ) = a ( x , h ) e^{i \varphi ( x ) / h },
\end{equation*}
where $\varphi$ is a generating function of $\Lambda$ near $\rho$ and $a \in S ( m ( h ) )$.

\bibliographystyle{amsplain}
\providecommand{\bysame}{\leavevmode\hbox to3em{\hrulefill}\thinspace}
\providecommand{\MR}{\relax\ifhmode\unskip\space\fi MR }
\providecommand{\MRhref}[2]{%
  \href{http://www.ams.org/mathscinet-getitem?mr=#1}{#2}
}
\providecommand{\href}[2]{#2}

\end{document}